\documentclass[12pt]{amsart}

 \usepackage{amsfonts,graphics,amsmath,amsthm,amsfonts,amscd,amssymb,amsmath,latexsym,multicol,
 mathrsfs}
\usepackage{epsfig,url}
\usepackage{flafter}
\usepackage{fancyhdr}
\usepackage{hyperref}
\hypersetup{colorlinks=true, linkcolor=black}

\makeatletter

\def\jobis#1{FF\fi
  \def\predicate{#1}%
  \edef\predicate{\expandafter\strip@prefix\meaning\predicate}%
  \edef\job{\jobname}%
  \ifx\job\predicate
}

\makeatother

\if\jobis{proposal}%
\else
\fi

 \usepackage[matrix, arrow]{xy}

\DeclareMathOperator{\Supp}{Supp}

\DeclareMathOperator{\codim}{codim}

\DeclareMathOperator{\Pic}{Pic}

 \numberwithin{equation}{subsection}
 \numberwithin{footnote}{subsection}

 \newtheorem{cor}[subsection]{Corollary}
 \newtheorem{lem}[subsection]{Lemma}
 \newtheorem{prop}[subsection]{Proposition}
 \newtheorem{thm}[subsection]{Theorem}

{
    \newtheoremstyle{upright}%
        {8pt plus2pt minus4pt}%
        {8pt plus2pt minus4pt}%
        {\upshape}%
        {}%
        {\bfseries\scshape}%
        {}%
        {1em}%
        {}%
\theoremstyle{upright}

 \newtheorem{defn}[subsection]{Definition}
 
 \newtheorem{exa}[subsection]{Example}

 \newtheorem{rem}[subsection]{Remark}

}

 \newcommand{\Q}{\mathbb Q}
 \newcommand{\R}{\mathbb R}
 
 \newcommand{\bir}{\dashrightarrow}

\title{Log canonical pairs with boundaries containing ample divisors}
\author{Zhengyu Hu}
\date{\today}
\begin{document}
\maketitle

\begin{abstract}
Let $(X,\Delta)$ be a projective log canonical pair such that $\Delta \geq A$ where $A \geq 0$ is an ample $\R$-divisor. We prove that either $(X,\Delta)$ has a good minimal model or a Mori fibre space. Moreover, if $X$ is $\Q$-factorial, then any Log Minimal Model Program on $K_X+\Delta$ with scaling terminates. As an application we prove that a log Fano type variety $X$ with $\Q$-factorial log canonical singularities is a Mori dream space. This gives an affirmative answer to a conjecture by Cascini and Gongyo. 
\end{abstract}



\section{Introduction}

We work over an algebraically closed field of characteristic zero .

\textbf{Existence of good log minimal models.}
Let $(X,\Delta)$ be a Kawamata log terminal pair such that $\Delta \geq A$ where $A \geq 0$ is an ample $\R$-divisor. C. Birkar, P. Cascini, C. D. Hacon and J. McKernan [\ref{BCHM}] proved that either $(X,\Delta)$ has a good minimal model or a Mori fibre space. This is one of the most important results from birational geometry established in recent years which immediately implies the finite generation of the canonical ring $R(X,B)$ of a Kawamata log terminal pair $(X,B)$ and the existence of Kawamata log terminal log flips. Because the existence of good minimal models is crucial to the birational classification of algebraic varieties, algebraic geometers made a great effort trying to generalise the above result for years. C. Birkar [\ref{B-lc-flips}], C. D. Hacon and C. Xu [\ref{HX2}] independently proved the existence of log canonical log flips by using a deep result on ACC for log canonical thresholds from C. D. Hacon, J. McKernan and C. Xu [\ref{HMX}]. Later C. Birkar and Z. Hu [\ref{BH-II}] showed that a log canonical pair $(X,B)$ with the good augmented base locus $\mathbf{B}_+(K_X+B)$ (more precisely the terminology "good" here means that the augmented base locus $\mathbf{B}_+(K_X+B)$ does not contain any log canonical centre of $(X,B)$ has a good minimal model. 

The main purpose of this article is to generalise C. Birkar, P. Cascini, C. D. Hacon and J. McKernan [\ref{BCHM}]'s famous theorem on the existence of good log minimal models from Kawamata log terminal (klt for short) pairs to log canonical (lc for short) pairs. To this end we consider the following more general situation but it is more flexible to arrange an inductive argument. Here is the main result of this paper.

\begin{thm}\label{t-main-1}
Let $(X,\Delta= A+B)$ be a projective log canonical pair of dimension $n$ where $A$, $B\geq 0$ are $\R$-divisors and the given morphism $f:(X,\Delta) \rightarrow Z$ is surjective. Assume further that $K_X+\Delta \sim_{\mathbb{R}} 0/Z$ and that $A \sim_\R f^{\ast} A_Z$ is the pull-back of an ample $\mathbb{R}$-divisor $A_Z$ on $Z$. Then, either $(X,\Delta)$ has a good minimal model or a Mori fibre space.
\end{thm}

As a corollary we obtain a generalisation of [\ref{BCHM}] for log canonical pairs.

\begin{cor}[= Corollary \ref{t-main-cor}]\label{t-main-cor'}
Let $(X,\Delta)$ be a projective log canonical pair such that $\Delta \geq A$ where $A \geq 0$ is an ample $\R$-divisor. Then, either $(X,\Delta)$ has a good minimal model or a Mori fibre space.  In particular, if $K_X+\Delta$ is $\Q$-Cartier, then the divisorial ring $R(X,\Delta)$ is finitely generated.
\end{cor}

\textbf{Terminations of log flips with scaling.}
[\ref{BCHM}] shows that the existence of good log minimal models (in a neighborhood of a boundary) is somehow equivalent to the terminations of log flips with scaling of some ample divisor. To investigate this a natural approach is to study the geography of log minimal models. We refer to [\ref{BCHM}], [\ref{Cascini-Lazic}], [\ref{Corti-Lazic}], [\ref{Kaw-3}], [\ref{Kaw-4}] for more details.

\begin{cor}[= Corollary \ref{cor-geography}, Geography of weak lc models]
	Let $X$ be a projective normal variety, and let $\{ \Delta_1$, $\Delta_2$, $\ldots$, $\Delta_r \}$ be a set of $\Q$-Weil divisors such that $(X,\Delta_i)$ is log canonical for each $1\le i \le r$. Let $\mathcal{P}$ be the rational polytope defined by $\{ \Delta_i \}$. Given a surjective morphism $f:X \rightarrow Z$. Assume further that $K_X+\Delta_i \sim_{\mathbb{R}} 0/Z$ for every index $i$ and that $\Delta_i \ge A$ where $A\sim_\R f^{\ast} A_Z$ is the pull-back of an ample $\mathbb{R}$-divisor $A_Z$ on $Z$. Then, the subset
	$$
	\mathcal{E}:=\{\Delta \in \mathcal{P}| \text{ $K_X+\Delta$ is pseudo-effective.}  \}
	$$
	is a rational polytope contained in $\mathcal{P}$. Moreover, $\mathcal{E}$ admits a finite rational polyhedral decomposition $\mathcal{E}= \bigcup_k \mathcal{E}_k$ satisfying: 
	
	$\bullet$ $\dim  \mathcal{E}_k =\dim  \mathcal{E}$ for every $k$;
	
	$\bullet$ $\dim  \mathcal{E}_k \bigcap \mathcal{E}_l  <\dim  \mathcal{E}$ for every $l\neq k$;
	
	$\bullet$ for any two divisors $\Delta$, $\Delta' \in \mathcal{E}$, $\Delta$ and $\Delta'$ belong to some same chamber if and only if there exists a normal variety $Y$ such that $(Y,\Delta_Y)$ and $(Y,\Delta_Y')$ are weak log canonical models of $(X,\Delta)$ and $(X,\Delta')$, where $\Delta_Y,\Delta'_{Y}$ are birational transforms of $\Delta,\Delta'$ respectively.
\end{cor}

\begin{thm}[= Theorem \ref{t-main-2}]
Let $(X,\Delta)$ be a $\Q$-factorial projective log canonical pair such that $\Delta \geq A$ where $A \geq 0$ is an ample $\R$-divisor. Then, any LMMP on $K_X+\Delta$ with scaling terminates.
\end{thm}

As immediate consequences we obtain the following corollaries.

\begin{cor}[= Corollary \ref{w-Mori}]
Let $(X,B)$ be a $\Q$-factorial projective log canonical pair such that $K_X+B$ is not pseudo-effective. Then, any LMMP on $K_X+B$ with scaling of some ample divisor terminates with a weak Mori fibre space. 
\end{cor}

\textbf{Varieties of log Fano type.}
Another consequence of Theorem \ref{t-main-1} and Corollary \ref{t-main-cor'} is the following result on varieties of log Fano type which also generalises the klt version in [\ref{BCHM}].
\begin{cor}[= Theorem \ref{coxlcfano}]
Let $X$ be a $\Q$-factorial projective variety. Assume there is a boundary $\Delta$ such that $(X,\Delta)$ is log canonical and that $-(K_X+B)$ is ample. Then a Cox ring of $X$ is finitely generated, and hence $X$ is a Mori dream space.
\end{cor}

\subsection*{Ackowledgement}
The main part of this work (Section 4 and Section 5) was written during the spring of 2014 when I was a research associate at the University of Cambridge. I am very grateful to Professor Caucher Birkar for his valuable advice and support. I would like to thank Professors Yoshinori Gongyo and Jinhyung Park for pointing out that the main result of this work implies a conjecture of Cascini and Gongyo. In particular Professor Jinhyung Park generously shared his proof (Section \ref{appendix}. Appendix) with me. The rest part of this paper was mostly written when I visited Institute for Mathematical Sciences, National University of Singapore on January, 2017. I would like to thank Professor De-Qi Zhang for his hospitality. I wish to express my deep gratitude to Dr. Kenta Hashizume for valuable comments and kindly pointing out a mistake in the previous version of the proof of Theorem \ref{t-main-2}. I would also like to thank Dr. Jingjun Han for his comments. \\

\section{Preliminaries}\label{preliminaries}

Let $k$ be an algebraically closed field of characteristic zero fixed throughout the paper. A divisor means an $\R$-Cartier $\R$-Weil
divisor. A divisor $D$ over a normal variety $X$ is a divisor on a birational model of $X$. A birational map $X \dashrightarrow Y$ is a \emph{birational contraction} if its inverse map contracts no divisor. \\

\textbf{Pairs.}
A \emph{pair} $(X/Z,B)$ consists of normal quasi-projective varieties $X$, $Z$, an $\R$-divisor $B$ on $X$ with
coefficients in $[0,1]$ such that $K_X+B$ is $\mathbb{R}$-Cartier and a projective morphism $X\rightarrow Z$. If $Z$ is a point or $Z$ is unambiguous in the context, then we simply denote a pair by $(X,B)$.
For a prime divisor $D$ on some birational model of $X$ with a
nonempty centre on $X$, $a(D,X,B)$
denotes the log discrepancy. For definitions and standard results on singularities of pairs
we refer to [\ref{Kollar-Mori}].\\

\textbf{Log minimal models and Mori fibre spaces.}
A projective pair $(Y/Z,B_Y)$ is a \emph{log birational model} of a projective pair $(X/Z,B)$ if we are given a birational map
$\phi\colon X\bir Y$ and $B_Y=B^\sim+E$ where $B^\sim$ is the birational transform of $B$ and
$E$ is the reduced exceptional divisor of $\phi^{-1}$, that is, $E=\sum E_j$ where $E_j$ are the
exceptional/$X$ prime divisors on $Y$.
A log birational model $(Y/Z,B_Y)$ is a  \emph{weak log canonical (weak lc for short) model} of $(X/Z,B)$ if

$\bullet$ $K_Y+B_Y$ is nef$/Z$, and

$\bullet$ for any prime divisor $D$ on $X$ which is exceptional/$Y$, we have
$$
a(D,X,B)\le a(D,Y,B_Y)
$$

 A weak lc model $(Y/Z,B_Y)$ is a \emph{log minimal model} of $(X/Z,B)$ if

$\bullet$ $(Y/Z,B_Y)$ is $\Q$-factorial dlt,

$\bullet$ the above inequality on log discrepancies is strict.

A log minimal model $(Y/Z, B_Y)$ is \emph{good} if $K_Y + B_Y$ is semi-ample$/Z$.\\

On the other hand, a log birational model $(Y/Z,B_Y)$  is called a \emph{weak Mori fibre space} of $(X/Z,B)$ if

$\bullet$ there is a $K_Y+B_Y$-negative extremal contraction $Y\to T$
with $\dim Y>\dim T$, and

$\bullet$ for any prime divisor $D$ (on birational models of $X$) we have
$$
a(D,X,B)\le a(D,Y,B_Y)
$$
 and strict inequality holds if $D$ is
on $X$ and contracted$/Y$.

 A weak Mori fibre space $(Y/Z,B_Y)$ is a \emph{Mori fibre space} of $(X/Z,B)$ if
 
$\bullet$ $(Y/Z,B_Y)$ is $\Q$-factorial dlt,\\

Note that our definitions of log minimal models and Mori fibre spaces are slightly different
from the traditional definitions in that we allow $\phi^{-1}$ to contract certain divisors.\\

\textbf{Log smooth models.}
A pair $(Y/Z,B_Y)$ is a \emph{log smooth model} of a pair $(X/Z,B)$ if there exists a birational morphism $\pi: Y \rightarrow X$ such that 

$\bullet$ $(Y/Z,B_Y)$ is log smooth,

$\bullet$ $\pi_\ast B_Y=B$,

$\bullet$ $a(E,Y,B_Y)=0$ for every exceptional$/X$ prime divisor on $Y$.

It is obvious that if $(X,B)$ is lc, then $(Y,B_Y)$ is dlt. In this case, it is easy to calculate that $a(D,X,B)\ge a(D,Y,B_Y)$ for any prime divisor $D$ over $X$. Moreover, a log minimal model of $(Y/Z,B_Y)$ is also a log minimal model of $(X/Z,B)$.   \\

\textbf{Ample models and log canonical models.}
Let $D$ be a divisor on a normal variety $X$ over $Z$. A normal variety $T$ is the \emph{ample model}$/Z$ of $D$ if we are given a rational map $\phi: X \bir T$ such that there exists a resolution $X \xleftarrow{p} X' \xrightarrow{q} T$ with 

$\bullet$ $q$ being a contraction,

$\bullet$ $p^\ast D \sim_\R q^\ast D_T+E$ where $D_T$ is an ample$/Z$ divisor and $E\ge 0$, and

$\bullet$ for every divisor $B \in |p^\ast D/Z|_\R$, then $B \geq E$. 

Note that the ample model is unique if it exists. The existence of the ample model is equivalent to saying that the divisorial ring $R(D)$ is a finitely generated $\mathcal{O}_Z$-algebra when $D \ge 0$ is $\Q$-Cartier. \\

A normal variety $T$ is the \emph{log canonical model}$/Z$ (\emph{lc model} for short) $T$ of a pair $(X/Z,B)$ if it is the ample model$/Z$ of $K_X+B$. The existence of the lc model of a klt pair is guaranteed by [\ref{BCHM}], and the existence of the lc model of an lc pair is still an open question. It is known that proving the existence of the lc models of lc pairs $(X,B)$ with $K_X+B$ being big is equivalent to proving the abundance conjecture for klt pairs.\\

\textbf{Nakayama-Zariski decompositions.}
Nakayama [\ref{Nakayama}] defined a decomposition $D=P_\sigma(D)+N_\sigma(D)$ for any pseudo-effective
$\R$-divisor $D$ on a smooth projective variety. We refer to this as the Nakayama-Zariski decomposition.
We call $P_\sigma$ the positive part and $N_\sigma$ the negative part. We can
extend it to the singular case as follows.
Let $X$ be a normal projective variety and $D$ be a pseudo-effective $\R$-Cartier divisor on $X$. We define $P_\sigma(D)$
by taking a resolution $f\colon W\to X$ and letting $P_\sigma(D):=f_*P_\sigma(f^*D)$. A divisor $D$ is \emph{movable} if $D= P_\sigma(D)$, that is, $D \in \overline{\mathrm{Mov}(X)}$. Given a normal variety $X/Z$ and a pseudo-effective divisor $D/Z$, we define the relative Nakayama-Zariski decomposition $D=P_\sigma(D/Z)+N_\sigma(D/Z)$ in a similar way. \\

\textbf{Asymptotic vanishing orders.}
We review some basic definitions developed from [\ref{Nakayama}] and [\ref{Bou}]. Suppose that $D$ is a divisor on a normal variety $X$ with $\kappa(D) \geq 0$ and $v=\mathrm{ord}_\Gamma$ is a discrete valuation. We define the \emph{asymptotic vanishing order} of $D$ along $v$ as
$$
v(\|D\|):= \liminf\limits_{L \in |mD|} \frac{1}{m} v(L).
$$
and the \emph{asymptotic fixed part} as $\mathrm{Fix}(\|D\|):=\sum_\Gamma \mathrm{ord}_\Gamma (\|D\|) \Gamma$ where $\Gamma$ runs over all prime divisors on $X$.

Fix an ample divisor $A$, we define the \emph{numerical asymptotic vanishing order} of a pseudo-effective divisor $D'$ along $v$ as
$$
\sigma_v(\|D'\|):= \lim\limits_{\epsilon \downarrow 0} v(\|D'+\epsilon A\|)
$$
[\ref{Nakayama}] verifies that the definition above is independent of the choice of $A$. It is obvious that the asymptotic vanishing order and the numerical asymptotic vanishing order coincide when $D$ is big. By an easy calculation one deduces that the negative part of the Nakayama-Zariski decomposition of $D$ has the coefficients which are numerical asymptotic vanishing orders: $N_\sigma(D) =\sum_\Gamma \sigma_\Gamma(\|D\|) \Gamma$.  \\

\textbf{Stable base loci and restricted base loci.}
Let $D$ be a pseudo-effective divisor on a normal projective variety $X$ over $Z$. The $\R$-stable base
locus of $D$ is defined as
$$
\mathbf{B}(|D/Z|_\R):= \bigcap \{ \mathrm{Supp}D'|D'\geq 0 \mathrm{~and~}D'\sim_\R D/Z \}.
$$ 
When $D$ is not $\R$-linearly equivalent to an effective divisor, we use the convention
that $\mathbf{B}(|D/Z|_\R)= X$. 

Similarly we define the $\Q$-stable base locus as $$
\mathbf{B}(|D/Z|_\Q):=\bigcap \{ \mathrm{Supp}D'|D'\geq 0 \mathrm{~and~}D'\sim_\Q D/Z \}.
$$ 

Next we define the restricted base locus of $D$ as
$$
\mathbf{B}_-(D/Z):= \bigcup_A \mathbf{B}(|(D+A)/Z|_\R)
$$
where the union is taken over all ample divisors $A$ on $X$ over $Z$.
\\

\textbf{LMMP with scaling.}
Let $(X/Z,B+C)=(X_1/Z, B_1 + C_1)$ be an lc pair such
that $K_{X_1} + B_1 + C_1$ is nef$/Z$, $B_1 \ge 0$, and $C_1 \ge 0$ is $\R$-Cartier. Suppose
that either $K_{X_1} + B_1$ is nef$/Z$ or there is an extremal ray $R_1/Z$ such that
$(K_{X_1} + B_1) \cdot R_1 < 0$ and $(K_{X_1} + B_1 + \lambda_1 C_1) \cdot R_1 = 0$ where
$$\lambda_1 := \inf \{t \ge 0 | K_{X_1} + B_1 + t C_1 \mathrm{~is~ nef}/Z \}.$$
Now, if $K_{X_1} + B_1$ is nef$/Z$ or if $R_1$ defines a Mori fibre structure, we stop.
Otherwise assume that $R_1$ gives a divisorial contraction or a log flip $X_1 \dashrightarrow X_2$.
We can now consider $(X_2/Z, B_2 + \lambda_1 C_2)$ where $B_2 + \lambda_1 C_2$ is the birational
transform of $B_1 + 
\lambda_1 C_1$ and continue. That is, suppose that either $K_{X_2} + B_2$
is nef$/Z$ or there is an extremal ray $R_2/Z$ such that $(K_{X_2} + B_2) \cdot R_2 < 0$ and
$(K_{X_2} + B_2 + \lambda_2 C_2) \cdot R_2 = 0$ where
$$\lambda_2 := \inf \{t \ge 0 | K_{X_2} + B_2 + t C_2 \mathrm{~is~ nef}/Z \}.$$
By continuing this process, we obtain a sequence of numbers $λ\lambda_i$ and a special
kind of LMMP$/Z$ which is called the LMMP$/Z$ on $K_{X_1} + B_1$ with scaling of $C_1$.
Note that by definition $\lambda_i \ge \lambda_{i+1}$ for every $i$, and we usually put $\lambda = \lim\limits_{i \rightarrow \infty} \lambda_i$.\\

\textbf{Decreasing the coefficients appeared in LMMP with scaling.}
Let $(X/Z,B+C)$ be as above. Assume further that $(X/Z,B)$ is $\Q$-factorial dlt. We suppose that for each $0< t \le 1$, $(X/Z,B+tC)$ has a good minimal model. Then, we can run an LMMP$/Z$ on $K_X+B$ with scaling of $C$ such that $\lambda = \lim\limits_{i \rightarrow \infty} \lambda_i=0$. To see this, we assume that $\lambda>0$. By [\ref{B-lc-flips}, Theorem 1.9], after finitely many steps we reach a model $(X_i/Z,B_i+\lambda C_i)$ on which $K_{X_i}+B_i +\lambda C_i$ is semi-ample$/Z$. Let $(Y/Z,B_Y)$ be the lc model of $(X_i/Z,B_i+\lambda C_i)$, and denote by $g: X/Z \rightarrow Y$. Now we instead run an LMMP$/Y$ on $K_{X_i}+B_i$ with scaling of $\lambda C_i$ which ends with a good minimal model $(X_j/Y,B_j)$ on which $K_{X_j}+B_j$ is semi-ample$/Y$ by Theorem \ref{lc-flips-1}. We therefore obtain that $\lambda_j< \lambda$. By continuing this process, we obtain the desired LMMP$/Z$.  \\

\textbf{Lifting a sequence of log flips with scaling.}
Given an LMMP with scaling which consists of a sequence of log flips $X_j \dashrightarrow X_{j+1}/Z_j$. Let $(X_1'/Z_1, B_1')$ be a $\Q$-factorial dlt
blow-up of $(X_1/Z, B_1)$ and $C_1'$ the pullback of $C_1$. Since $K_{X_1} + B_1 + \lambda_1 C_1 \equiv 0/Z_1$, we get $K_{X_1′}+ B_1' + \lambda_1 C_1' \equiv 0/Z_1$.

Run an LMMP$/Z_1$ on $K_{X_1'} + B_1'$ with scaling of an ample$/Z_1$ divisor
which is automatically also an LMMP$/Z_1$ on $K_{X_1'}+ B_1'$ with scaling of $\lambda_1 C_1'$.
Assume that this LMMP terminates with a log minimal model $(X_2'/Z_1, B_2')$. By construction, $(X_2/Z_1,B_2)$ is the lc model of $(X_1'/Z_1, B_1')$ and hence $(X_2'/Z_1, B_2')$ is a $\Q$-factorial dlt blow-up of $(X_2/Z_1, B_2)$. We can continue this process on $(X_2'/Z_2, B_2')$ and lift the original sequence to an LMMP on $K_{X_1'}+ B_1'$ with scaling of $\lambda_1 C_1'$. \\

\textbf{ACC for log canonical thresholds.}
Theorems on ACC proved in [\ref{HMX}] are crucial to the proof of our main result. We present them here for convenience. Let us recall the log canonical threshold of an effective divisor $M$ with respect to a log pair $(X,\Delta)$, denoted by $\mathrm{lct}(X,\Delta;M)$, is defined as 
$$
\mathrm{lct}(X,\Delta;M):=\sup \{t\in \R | (X,\Delta+t M) \mathrm{~is~log~canonical}\}.
$$
\begin{thm}[ACC for log canonical thresholds, see $\mathrm{[\ref{HMX}, Theorem ~1.1]}$]\label{ACC-1}
Fix a positive integer $n$, a set $I \subset [0, 1]$ and a set $J \subset \R_{>0}$, where
$I$ and $J$ satisfy the DCC. Let $\mathfrak{T}_n(I)$ be the set of log canonical pairs
$(X, \Delta)$, where $X$ is a variety of dimension $n$ and the coefficients of $\Delta$
belong to $I$. Then the set
$$
\{\mathrm{lct}(X, \Delta; M) | (X, \Delta) \in \mathfrak{T}_n(I), \mathrm{~the ~coefficients ~of~} M \mathrm{~ belong~ to~} J\}
$$
satisfies the ACC.
\end{thm}

\begin{thm}[ACC for numerically trivial pairs, see $\mathrm{[\ref{HMX}, Theorem  ~D]}$]\label{ACC-2}
Fix a positive integer $n$ and a set $I \subset [0, 1]$, which satisfies the
DCC.
Then there is a finite set $I_0 \subset I$ with the following property: \\
If $(X, \Delta)$ is a log canonical pair such that \\
(i) $X$ is projective of dimension $n$, \\
(ii) the coefficients of $\Delta$ belong to $I$, and \\
(iii) $K_X + \Delta$ is numerically trivial, \\
then the coefficients of $\Delta$ belong to $I_0$.
\end{thm}

\textbf{A special LMMP.}
Given a $\Q$-factorial dlt pair $(X,B)$ and a divisor $P$ with its support $\mathrm{Supp}P= \lfloor B\rfloor$, Birkar [\ref{B-lc-flips}] defined a sequence of birational maps as follows. Here we put everything into the situation of $\R$-divisors for our use later. Pick a decreasing sequence of sufficiently small positive real numbers $\epsilon_1 >\epsilon_2 > \cdots >\epsilon_j > \cdots$ such that $\lim_{j \rightarrow \infty} \epsilon_j =0$. Obviously, the pair $(X,B- \epsilon_j P)$ is klt provided that $\epsilon_1$ is sufficiently small. For each $j$, if we can run an LMMP on $K_X+B -\epsilon_j P$ which terminates with a good minimal model $(X_j, B_j-\epsilon_j P_j)$ where $B_j$ and $P_j$ are birational transforms of $B$ and $P$ respectively, then $(X_j,B_j)$ is $\Q$-factorial lc by Theorem \ref{ACC-1}. Since $\epsilon_1$ can be chosen arbitrary small, we can assume that all $X_j$'s are isomorphic in codimension one.

Since $K_{X_1}+B_1- \epsilon_1 P_1$ is semi-ample, 
$$
K_{X_1}+B_1- \epsilon_1 P_1 \sim_\R C_1
$$
for some $C_1$ so that $(X_1,B_1+C_1)$ is lc. If $t_j=\frac{\epsilon_j}{\epsilon_1-\epsilon_j}$ and $j \neq 1$, then 
$$
K_{X_1}+B_1+t_j C_1 \sim_\R (1+t_j)(K_{X_j}+ B_1 -\epsilon_j P_1).
$$
Next we can run an LMMP on $K_{X_2}+B_2$ with scaling of $t_2 C_2$. Because $(X_j,B_j+t_jC_j)$ is a good minimal model of $(X_2,B_2+t_jB_2)$, we can assume that $(X_j,B_j)$ appears in the process of this LMMP for every $j$ by [\ref{B-lc-flips}, Theorem 1.9]. In other words, the sequence of birational maps 
$$
(X_2,B_2) \rightarrow \cdots \rightarrow (X_j,B_j) \rightarrow \cdots
$$ 
can be embedded into a sequence of log flips. In particular, if the LMMP above does not intersect with $\lfloor B_j \rfloor$ for $j \gg 0$, then the program terminates.    \\

\textbf{On special termination.}
Assume that we are given an LMMP
with scaling which consists of only a sequence $X_i \dashrightarrow
X_{i+1}/Z_i$ of log flips, and that $(X_1/Z, B_1)$ is $\Q$-factorial dlt. Assume $\lfloor B_1 \rfloor \neq  0$
and pick a component $S_1$ of $\lfloor B_1 \rfloor$. Let $S_i \subset X_i$ be the birational transform of
$S_1$ and $T_i$ the normalisation of the image of $S_i$
in $Z_i$. Using standard special
termination arguments, we will see that termination of the LMMP near $S_1$ is
reduced to termination in lower dimensions. It is well-known that the induced
map $S_i \dashrightarrow S_{i+1}/T_i$
is an isomorphism in codimension one if $i \gg 0$ (cf. [\ref{Fujino-2}]).
So, we could assume that these maps are all isomorphisms in codimension one.
Put $K_{S_i}+B_{S_i}
:= (K_{X_i}+B_i)|_{S_i}$. In general, $S_i \dashrightarrow S_{i+1}/T_i$
is not a $(K_{S_i}+B_{S_i})$
-flip.
To apply induction, we note that $S_i \dashrightarrow S_{i+1}/T_i$ can be connected by a sequence of $(K_{S_i}+B_{S_i})$-flips (see [\ref{B-lc-flips}, Remark 2.10]).\\

\textbf{Very exceptional divisors.}
Let $f : X \rightarrow Y$ be a
contraction of normal varieties, $D$ a divisor on $X$, and $V \subset X$ a closed
subset. We say that $V$ is vertical over $Y$ or $f$-vertical if $f(V)$ is a proper subset of $Y$. We
say that $D$ is \emph{very exceptional}$/Y$ if $D$ is vertical$/Y$ and for any prime divisor
$P$ on Y there is a prime divisor $Q$ on $X$ which is not a component of $D$ but
$f(Q) = P$, i.e. over the generic point of $P$ we have $\mathrm{Supp} f^\ast P \nsubseteq \mathrm{Supp} D$. In some literatures such $D$ is also said to be \emph{$f$-degenerate} (cf. [\ref{Gongyo-Lehmann}]). Here we cite some lemmas for later use.

\begin{lem}[$\mathrm{[}$\ref{B-lc-flips}, Lemma 3.2$\mathrm{]}$, cf. Shokurov $\mathrm{[}$\ref{Shokurov}, Lemma 3.19$\mathrm{]}$]\label{lem-exc-1}
Let $f : X \rightarrow Y$ be a contraction
of normal varieties, projective over a normal affine variety $Z$. Let $A$ be an
ample$/Z$ divisor on $Y$ and $F = f^\ast A$. If $E \ge 0$ is a divisor on $X$ which is
vertical$/Y$ and such that $mE = \mathrm{Fix}(mF + mE)$ for every integer $m \gg 0$, then
$E$ is very exceptional$/Y$.
\end{lem}

\begin{lem}[$\mathrm{[}$\ref{B-lc-flips}, Theorem 3.4$\mathrm{]}$]\label{lem-exc-2}
Let $(X/Z, B)$ be a $\Q$-factorial dlt pair such
that $K_X + B \sim_\R M/Z$ with $M \ge 0$ very exceptional$/Z$. Then, any LMMP$/Z$
on $K_X + B$ with scaling of an ample$/Z$ divisor terminates with a model $Y$ on
which $K_Y + B_Y \sim_\R M_Y = 0/Z$.
\end{lem}

\begin{lem}[$\mathrm{[}$\ref{B-lc-flips}, Theorem 3.5$\mathrm{]}$]\label{lem-exc-3}
Let $(X/Z, B)$ be a $\Q$-factorial dlt pair such that $X \rightarrow Z$ is
birational, and $K_X + B \sim_\R M = M_+ - M_-/Z$ where $M_+$, $M_- \ge 0$ have
no common components and $M_+$ is exceptional$/Z$. Then, any LMMP$/Z$ on
$K_X + B$ with scaling of an ample$/Z$ divisor contracts $M_+$ after finitely many
steps.
\end{lem}

If $\codim f(D) \ge 2$, then $D$ is very exceptional. In this case we say $D$ is \emph{$f$-exceptional}. 
\\

\textbf{Toroidal reductions and equidimensional morphisms.}
We will need a few results of toroidal reductions in Section \ref{gen-trivial-pairs}.

\begin{thm}[$\mathrm{[}$\ref{AK}, Theorem 2.1$\mathrm{]}$]\label{thm-toroidal}
	Let $f: X\to Z$ be a projective surjective morphism with geometrically integral generic fibre, and assume $Z$ integral. Let $B \subset X$ be a proper closed subscheme. There exists a diagram as follows :
$$\begin{array}{lclcl} 
U_X & \subset & X' &\stackrel{m_X}{\to} &X \\
\downarrow & & \downarrow f' & & \downarrow f\\
 U_Z & \subset & Z'
& \stackrel{m_Z}{\to} &Z \end{array}
$$
such that $m_X$ and $m_Z$ are birational projective morphisms, $X'$ and $Z'$ are nonsingular, the inclusions on the left are toroidal embeddings, and such that:

1. $f'$ is toroidal.

2. Let $B'=m_X^{-1} B$. Then $B'$ is a simple normal crossings divisor, and $B' \subset X'\backslash U_{X'}$.

\end{thm}

\begin{prop}[$\mathrm{[}$\ref{AK}, Proposition 4.4$\mathrm{]}$]\label{prop-toroidal}
	Let $U_X\subset X$ and $U_B \subset B$ be toroidal embeddings with
	polyhedral complexes $\Delta_X$ and $\Delta_B$ respectively, and
	assume that $B$ is nonsingular. Let $ f: X\to B$ be a surjective
	toroidal morphism. Then there exist projective subdivisions
	$\Delta_X'$ of $\Delta_X$ and $\Delta_B'$ of $\Delta_B$ with
	$\Delta_B'$ nonsingular, such that the induced map $f':X'\to B'$ is an
	equidimensional toroidal morphism.
\end{prop}

\begin{lem}\label{lem-toroidal}
	Let $\pi: X' \to X$ be a birational projective morphism to a smooth variety. Then, there exists a small $\Q$-factorialisation $\tau : \overline{X'} \to X'$. 
\end{lem}
\begin{proof}
	We can assume that $X,X'$ are projective. Let $\{ E_i \}$ be the set of exceptional$/X$ divisors on $X'$. Pick an ample $\Q$-divisor $A$ such that $(X,A)$ is klt and $a(E_i,X,A) \le 1$ for each $i$. We write $K_{X'} +A' =\pi^*(K_X+A)$. Since $(X',A')$ is klt, by a standard argument we obtain a small $\Q$-factorialisation.
\end{proof}

\textbf{A lemma on weak lc models.}
We prove the following easy lemma for the use of an argument in Section \ref{vertical-case}.

\begin{lem}\label{weak-lc-model-lem}
Let $(X,\Delta)$, $(W,\Delta_W)$ be projective $\Q$-factorial dlt pairs and $g:W \dashrightarrow X$ be a birational map. Assume that 

$\bullet$ $K_{W}+\Delta_{W}$ is semi-ample;

$\bullet$ $a(\Gamma,X,\Delta) \le a(\Gamma,W,\Delta_W)$ for any prime divisor $\Gamma$ over $X$;

$\bullet$ if $\Gamma$ is a prime divisor on $W$, then $a(\Gamma,X,\Delta) + \sigma_\Gamma(K_X+\Delta) \ge a(\Gamma,W,\Delta_W)$.

Then, $(X,\Delta)$ has a good minimal model.
\end{lem}

\begin{proof}
Let $W \xleftarrow{p} \widetilde{W} \xrightarrow{q} X$ be a common log resolution. If we write $K_{\widetilde{W}}+\Delta_{\widetilde{W}}= q^\ast (K_X+\Delta)) +E$ where $\Delta_{\widetilde{W}}$, $E \ge0$ have no common component and $E$ is exceptional$/X$, then a log minimal model of $(\widetilde{W},\Delta_{\widetilde{W}})$ is also a log minimal model of $(X,\Delta)$. Moreover,  if $\Gamma$ is a prime divisor on $W$, then $\sigma_\Gamma(K_X+\Delta)=\sigma_\Gamma(K_{\widetilde{W}}+\Delta_{\widetilde{W}})$ as $a(\Gamma,X,\Delta)=a(\Gamma,\widetilde{W},\Delta_{\widetilde{W}}) \le a(\Gamma,W,\Delta_W) \le 1$. Since $a(\Gamma,X,\Delta) \le a(\Gamma,W,\Delta_W)$, we have $p_\ast \Delta_{\widetilde{W}} \ge \Delta_W$. Put $G=p_\ast \Delta_{\widetilde{W}} -\Delta_W$ and $\widetilde{G}= p_\ast^{-1} G$. If we write $\Delta_{\widetilde{W}}'=\Delta_{\widetilde{W}} -\widetilde{G}$, then we have $a(\Gamma,W,\Delta_W)= a(\Gamma,\widetilde{W},\Delta_{\widetilde{W}}')$ for any prime divisor $\Gamma$ on $W$. Because $K_W+\Delta_W$ is nef, by Negativity Lemma $a(\Gamma,W,\Delta_W)\ge a(\Gamma,\widetilde{W},\Delta_{\widetilde{W}}')$ and hence $(W,\Delta_W)$ is a weak lc model of $(\widetilde{W},\Delta_{\widetilde{W}}')$. 

Because $a(\Gamma,X,\Delta) + \sigma_\Gamma(K_X+\Delta) \ge a(\Gamma,W,\Delta_W)$, for every prime divisor on $W$, we deduce $\widetilde{G} \le N_\sigma(K_{\widetilde{W}}+\Delta_{\widetilde{W}})$. Therefore $(\widetilde{W},\Delta_{\widetilde{W}}')$ has a log minimal model which in turn implies that $(\widetilde{W},\Delta_{\widetilde{W}})$ has a log minimal model by [\ref{B-lc-flips}, Corollary 3.7] and [\ref{BH-I}] since $\widetilde{G} \le N_\sigma(K_{\widetilde{W}}+\Delta_{\widetilde{W}})$. Hence $(X,\Delta)$ has a log minimal model $(Y,B_Y)$. Because $K_{W}+\Delta_{W}$ is semi-ample, by [\ref{B-lc-flips}, Remark 2.7], $K_Y+\Delta_Y$ is semi-ample. 
\end{proof}

\vspace{0.3cm}
\section{Minimal models for generically trivial pairs}\label{gen-trivial-pairs}

We need some results from C. Birkar [\ref{B-lc-flips}, Theorem 1.1, Theorem 1.4 and Theorem 1.7], C.-D. Hacon and C. Xu [\ref{HX}, Corollary 1.5], [\ref{HX2}, Theorem 1.6] (cf. Fujino and Gongyo [\ref{FG2}]) with a generalisation for $\R$-divisors. K. Hashizume [\ref{Hashizume}] informed me that he also obtain the same result with a different approach.

\begin{thm}\label{lc-flips-1}
	Let $(X/Z, B +A)$ be an lc pair where $A \geq 0$
	is $\R$-Cartier, and the given morphism $f : X \rightarrow Z$ is surjective. Assume further
	that $K_X + B + A \sim_\R 0/Z$. Then, 
	
	(1) $(X/Z, B)$ has a Mori fibre space or a log minimal model $(Y/Z, B_Y)$,
	
	(2) if $K_Y + B_Y$ is nef$/Z$, then it is semi-ample$/Z$.
\end{thm}

\begin{thm}\label{lc-flips-2}
	Let $(X/Z, B)$ be an lc
	pair and the given morphism $f : X \rightarrow Z$
	is surjective. Assume further that $K_X + B \sim_\R 0$ over some non-empty open
	subset $U \subseteq Z$, and if $\eta$ is the generic point of an lc centre of $(X/Z, B)$, then
	$f(\eta)  \in U$. Then,
	
	(1) $(X/Z, B)$ has a log minimal model $(Y/Z, B_Y )$,
	
	(2) $K_Y + B_Y$ is semi-ample$/Z$.
\end{thm}

\begin{thm}\label{lc-flips-3}
	Let $(X/Z, B)$ be a $\Q$-factorial dlt pair and $G \subseteq \lfloor B \rfloor$ be a reduced divisor. Suppose that
	
	$\bullet$ $K_X + B$ is nef$/Z$,
	
	$\bullet$ $(K_X + B)|_S$ is semi-ample$/Z$ for each component $S$ of $G$,
	
	$\bullet$ $K_X + B - \epsilon P$ is semi-ample$/Z$ for some divisor $P \geq 0$ with $\mathrm{Supp} P = G$
	and for any sufficiently small real number $\epsilon > 0$.
	
	Then, $K_X + B$ is semi-ample$/Z$.	
\end{thm}

\begin{rem}
We only need the case when $G = \lfloor B \rfloor$ in this article. However, we prove Theorem \ref{lc-flips-3} in full generality since the same argument works.	
\end{rem}

An important ingredient of the proof of Theorem \ref{lc-flips-3} is the following Diophantine approximation.

\begin{lem}[see $\mathrm{[}$\ref{BCHM}, Lemma 3.7.7$\mathrm{]}$]\label{diophantine}
Let $\mathcal{C}$ be a rational polytope contained in a real vector
space $V$ of dimension $n$, which is defined over the rationals. Fix a
positive integer $k$ and a positive real number $\alpha$.
If $v \in \mathcal{C}$, then we may find vectors $v_1$, $v_2$, $\dots$, $\in \mathcal{C}$ and positive
integers $m_1$, $m_2$, $\dots$, $m_p$, which are divisible by $k$, such that $v$ is a convex
linear combination of the vectors $v_1$, $v_2$, $\dots$, $v_p$ and
\begin{align*}
\|v_i-v\|\leq \frac{\alpha}{m_i} \mathrm{~where~} \frac{m_iv_i}{k} \mathrm{~is~integral}
\end{align*}
where $\|v_i-v\|$ denotes the norm of $v_i-v$ in $\R^n$.
\end{lem}

The following lemma may be known to experts. We write a detailed proof for the reader's convenience.

\begin{lem}\label{semi-ample}
Let $f:X \rightarrow Z$ be a projective morphism of normal varieties and $D$ be a divisor on $X$. Assume that $X$ is $\Q$-factorial and $\mathbf{B}(|D/Z|_\R)= \emptyset$. Then, $D$ is semi-ample$/Z$. 
\end{lem}

\begin{proof}
By replacing $D$ with a member of $|D/Z|_\R$, we can assume that $D$ is effective. For simplicity we assume that $Z$ is a point. Let $\Sigma_i D_i$ be the irreducible
decompositions of $\mathrm{Supp} D $. For every $i$ there exists a divisor $M_i \sim_\R D$ such that $M_i$ does not contain $D_i$ in its support. More precisely, we write 
$$
D+ \sum_j a_{ij}(g_{ij})  =M_i
$$
where $a_{ij}$ are real coefficients and $g_{ij}$ are rational functions on $X$. Let $\Sigma_k M_{ik}$ be the irreducible
decompositions of $\mathrm{Supp} M_i $, $\mathcal{V}:= \oplus_i \R_{\geq 0} B_i$ and $\mathcal{W}_i= \oplus_k \R_{\geq 0} M_{ik}$. Moreover, let $\mathcal{R}_i =\oplus_j \R (g_{ij})$. It is easy to check that there is  a rational polytope 
\begin{align*}
\mathcal{L}_i \subset \{D' \in \mathcal{V}| D'+\mathcal{R}_i \mathrm{~intersects~}  \mathcal{W}_i \}
\end{align*}
containing $D$, and for every $\Q$-divisor $D' \in \mathcal{L} =\bigcap_i \mathcal{L}_i$, we have  $$\mathbf{B}(|D'|_\Q) \subset \bigcup_i (D_i \bigcap \mathrm{Supp}(M_i)) $$ by [\ref{Nakayama}, Chapter II, 2.8 Lemma]. Note that $\dim D_i \bigcap \mathrm{Supp}(M_i) < \dim B_i$. Next we pick a divisor $M_i'$ for each $i$ such that $M_i'$ does not contain $B_i \bigcap \mathrm{Supp}(M_i)$ in its support. In a similar way we shrink $\mathcal{L}$ to a smaller polytope such that for every $\Q$-divisor $D' \in \mathcal{L} $, $\mathbf{B}(|D'|_\Q) \subset \bigcup_i (B_i \bigcap \mathrm{Supp}(M_i) \bigcap  \mathrm{Supp}(M_i')) $. By repeating this process we obtain a rational polytope $\mathcal{L}$ containing $D$ such that every $\Q$-divisor $D' \in \mathcal{L}$ is semi-ample. Hence there exists a finite number of semi-ample $\Q$-divisors $D'_l$ such that $D$ is a convex linear combination of $D'_l$.  Finally, by a standard argument we deduce that $D$ is semi-ample.
\end{proof}

\begin{proof}[Proof of Theorem \ref{lc-flips-3}]
If $B$ is a $\Q$-divisor, then the conclusion follows from a similar proof of [\ref{HX}, Corollary 1.5] with an aid of an injectivity theorem (cf. [\ref{Fujino}, Theorem 6.1], [\ref{Ambro-2}, Theorem 2.3], [\ref{Fujino-3}, Proposition 5.1.1], [\ref{Kollar}], etc.). We therefore assume that $B$ is an $\R$-divisor with some coefficient irrational. Let $S=\lfloor B\rfloor$ (by abuse of notation) and $\Sigma_i B_i$ be the irreducible
decompositions of $\mathrm{Supp} \Delta $ where $\Delta=\{B\}$ is the fractional part of $B$. We put $\mathcal{V}:=\oplus_i \R B_i$. It is easy to check that
$$
\mathcal{L} =\{\Delta' \in \mathcal{V} |(X, S+\Delta')\mathrm{~is~log ~canonical ~and~} K_X+S+\Delta' \mathrm{~is~nef}/Z  \}
$$
is a rational polytope containing $\Delta$.

Let $S_i$ be a component of $G$. Since $K_{S_i}+B_{S_i}= (K_X+B)|_{S_i}$ is semi-ample$/Z$, there is a morphism $g_i:S_i \rightarrow T_i$ where $T_i$ is the lc model of $(S_i/Z,B_{S_i})$. So, there is an ample$/Z$ divisor $A_i$ on $T_i$ and a finite number of rational functions $h_{ij}$ on $S_i$ such that
 $$
K_{S_i}+B_{S_i}+\Sigma_j a_{ij}(h_{ij})= g_i^\ast A_i \in  g_i^\ast \mathrm{Div}_\R(T_i/Z)
$$ where $a_{ij}$ are real coefficients. 
Let $N_i=\lfloor B_{S_i} \rfloor$ and $\Sigma_k B_{S_i,k}$ be the irreducible
decompositions of $\mathrm{Supp}\{ B_{S_i}\}$. We define $\mathcal{W}_i:=\oplus_k \R B_{S_i,k}$, $\mathcal{R}_i:=\oplus_j \R(h_{ij})$ and $\mathcal{A}_i \subset \mathrm{Div}_\R(T_i/Z)$ as a sufficiently small rational polytope containing $A_i$. It is easy to check that 
\begin{align*}
\mathcal{P}_i := \{ &\Delta_{S_i}' \in \mathcal{W}_i| (S_i, N_i +\Delta_{S_i}')\mathrm{~is~log ~canonical ~and~} \\
& K_{S_i}+N_i +\Delta_{S_i}'+ \mathcal{R}_i\mathrm{~intersects~}g_i^\ast \mathcal{A}_i \}
\end{align*}
is a rational polytope in by [\ref{B-II}, Proposition 3.2]. It follows that
\begin{align*}
\mathcal{L}_i := \{& \Delta' \in \mathcal{L}| \Delta_{S_i}' \in \mathcal{P}_i  \mathrm{~where~} \Delta_{S_i}' \mathrm{~ is~defined~as~} \\
& K_{S_i}+N_i+\Delta_{S_i}' =(K_X+S+\Delta')|_{S_i} \}
\end{align*}
is also a rational polytope containing $\Delta=\{B\}$. Let $\mathcal{L}'= \cap_i \mathcal{L}_i$ where $i$ runs over all components $S_i$ of $G$. By assumptions we deduce that $\mathcal{L}'$ is a rational polytope containing $\Delta$. 

Let $\Delta_k $'s be a finite set of $\Q$-divisors in $\mathcal{L}'$ such that $\Delta$ is a convex linear combination of $\Delta_k$'s, that is, $\Delta=  \Sigma \lambda_k \Delta_k$ and $\Sigma \lambda_k=1$. Since such $\Delta_k$'s can be chosen arbitrarily close to $\Delta$,  $K_{S_i}+N_i+\Delta_{k,S_i}$ is semi-ample$/Z$ for every $i$.

Let $\beta=\|\Delta \|$, let $d$ be an integer such that $dK_X$ and $d B_i$ are Cartier for all $i$, and pick $\alpha \ll 1- \beta$. By Lemma \ref{diophantine} we can further assume that 
\begin{align*}
\|\Delta_k-\Delta \|\leq \frac{\alpha}{m_k} \mathrm{~where~} \frac{m_k}{d}\Delta_k \mathrm{~is~integral}.
\end{align*}
Then, we use a similar argument from [\ref{HX}, proof of 1.5] (see also [\ref{KMM}, 7.4]). We write
\begin{align*}
m_k(K_X+S+\Delta_k)-G =& K_X+B +m_k(\Delta_k-\Delta) -G +(m_k-1)\epsilon P\\
 &+(m_k-1)(K_X+B-\epsilon P). 
\end{align*}
Since $(X,B +m_k(\Delta_k-\Delta) -G +(m_k-1)\epsilon P)$ is dlt for $0<\epsilon \ll 1$ and $K_X+B-\epsilon P$ is semi-ample$/Z$, by an injectivity theorem (cf. [\ref{Fujino}, Theorem 6.1], [\ref{Ambro-2}, Theorem 2.3], [\ref{Fujino-3}, Proposition 2.23], [\ref{Kollar}], etc.) we have that
$$
R^1f_\ast \mathcal{O}_X(m_k(K_X+S+\Delta_k)-G) \rightarrow R^1f_\ast \mathcal{O}_X(m_k(K_X+S+\Delta_k))
$$
is an injection where $f: X \rightarrow Z$ and hence 
$$f_\ast \mathcal{O}_X(m_k(K_X+S+\Delta_k)) \rightarrow f_\ast \mathcal{O}_G(m_k(K_G+(S-G)|_G+\Delta_{k,G}))
$$ 
is surjective where $\Delta_{k,G}$ is defined as $K_G+(S-G)|_G+\Delta_{k,G} =(K_X+S+\Delta_k)|_G$. We have the commutative diagram as follows.
$$
\xymatrix{
	f^\ast f_\ast \mathcal{O}_X(m_k(K_X+S+\Delta_k)) \ar[d]_{} \ar[r]_{}&  f^\ast f_\ast \mathcal{O}_G(m_k(K_G+(S-G)|_G+\Delta_{k,G})) \ar[d]^{}  &\\
	\mathcal{O}_X(m_k(K_X+S+\Delta_k))  \ar[r]_{}  & \mathcal{O}_G(m_k(K_G+(S-G)|_G+\Delta_{k,G}))    &
	} 
$$
Since the upper arrow and the right arrow are surjective, $m_k(K_X + S+\Delta_k)$ is relatively
globally generated along $G$ over $Z$ which in turn implies that the $\R$-stable base locus $\mathbf{B}(|K_X+B/Z|_\R)$ does not intersect with $G$. Since $K_X+B-\epsilon P$ is semi-ample$/Z$ for $0 <\epsilon \ll 1$, the $\R$-stable base locus $\mathbf{B}(|K_X+B/Z|_\R)$ must be contained in the support of $G$. Hence we obtain that $\mathbf{B}(|K_X+B/Z|_\R) =\emptyset$. We apply Lemma \ref{semi-ample} to conclude that $K_X+B$ is semi-ample$/Z$.
\end{proof}

The following lemma is very useful.

\begin{lem}\label{lc-m-g-model}
	Let $(X/Z,B)$ be an lc pair. Assume that $(X/Z,B)$ has the lc model and that $K_X+B$ is abundant$/Z$. Then, $(X/Z,B)$ has a good minimal model.
\end{lem}

\begin{proof}
	We first treat the case when $Z$ is a point. Since $(X,B)$ has the lc model, there is a log resolution $f:X' \longrightarrow X$ such that $f^\ast (K_X+B)=M+F$ where $M$ is semi-ample and $F$ is the asymptotic fixed part. Note that the abundance of $K_X+B$ implies that $F= N_\sigma (f^\ast(K_X+B))$ by [\ref{Lehmann}, Proposition 6.4]. Therefore, $K_X+B$ birationally has a Nakayama-Zariski decomposition with nef positive part. We immediately obtain that $(X,B)$ has a log minimal model $(Y,B_Y)$ according to [\ref{BH-I}]. We assume that the birational map $g:X' \dashrightarrow Y$ is a morphism, and it is obvious that $g^\ast (K_Y+B_Y)= P_\sigma (f^\ast(K_X+B))=M$ is semi-ample. It follows that the log minimal model $(Y,B_Y)$ is good.
	
	Next we prove the general case. Replacing $(X/Z,B)$ we can assume it is $\Q$-factorial dlt. Run an LMMP$/Z$ on $K_X+B$ with scaling of an ample$/Z$ divisor. From the argument above we reach a model $g:(Y,B_Y) \dashrightarrow T$ after finitely many steps where $T$ is the lc model of $(X/Z,B)$ and $(K_Y+B_Y)|_{Y_\eta} \sim_\R 0$ on the generic fibre $Y_\eta$ of $g$. Replacing $(Y/Z,B_Y)$ with some birational model, we can assume that $g$ is a morphism and write $K_Y+B_Y\sim_\R g^\ast A_T +F_Y$ where $A_T$ is an ample$/Z$ divisor and $F_Y \ge 0$ is vertical$/T$. By Lemma \ref{lem-exc-1}, $F_Y$ is very exceptional$/T$, and hence  by Lemma \ref{lem-exc-2} any LMMP$/T'$ on $K_{Y'}+B_{Y'}$ with scaling of an ample$/Z$ divisor contracts $F_{Y'}$ and terminates with a minimal model $g'': (Y'',B_{Y''})\rightarrow T$ on which $K_{Y''}+B_{Y''} \sim_\R g''^\ast A_T$. 
\end{proof}

The following lemma is [\ref{B-lc-flips}, Theorem 1.5] with a generalisation for $\R$-divisors.

\begin{lem}\label{lc-m-g-model-2}
Let $(X/Z, B)$ be a $\Q$-factorial dlt pair where  $f : X \rightarrow Z$ is projective. Assume further that $(X/Z,B)$ has the lc model$/Z$, and that $(K_X + B)|_{X_\eta} \sim_\R 0$ where $X_\eta$
is the generic fibre of $f$. Then, any LMMP$/Z$ on $K_X + B$ with scaling of an
ample$/Z$ divisor terminates with a good minimal model.
\end{lem}

\begin{proof}
We assume that $f$ is a contraction. Since $(K_X + B)|_{X_\eta} \sim_\R 0$ where $X_\eta$
is the generic fibre of $f$, $K_X+B$ is abundant$/Z$. We derive the conclusion from Lemma \ref{lc-m-g-model} .
\end{proof}

The argument for the following lemma is similar to [\ref{Hashizume}, Lemma 3.2].

\begin{lem}\label{klt-m-g-model}
	Let $(X/Z, B)$ be a $\Q$-factorial klt pair where  $f : X \rightarrow Z$ is projective. Assume further that $(K_X + B)|_{X_\eta} \sim_\R 0$ where $X_\eta$
	is the generic fibre of $f$. Then, any LMMP$/Z$ on $K_X + B$ with scaling of an
	ample$/Z$ divisor terminates with a good minimal model.
\end{lem}
\begin{proof}
	By replacing $(X,B)$ and $Z$ we assume $(X,B)$ is log smooth and $Z$ is smooth. Note that of course we lose the condition $(K_X + B)|_{X_\eta} \sim_\R 0$. However, we still have $(K_X + B)|_{X_\eta} \sim_\R N_\sigma((K_X + B)|_{X_\eta})$. Now by replacing $X,Z$ with suitable compactifications we can assume $X,Z$ are projective. By assumption we let $K_X+B \sim_\R D= D_h+D_v$ where $D_v \ge 0$ is vertical$/Z$ and each component of $D_h \ge 0$ is horizontal$/Z$. In particular, we have $D_h|_{X_\eta} = N_\sigma(D_h|_{X_\eta})$. By Theorem \ref{thm-toroidal} there exists a diagram as follows: 
	$$\begin{array}{lclcl} 
	U_X & \subset & X' &\stackrel{\pi_X}{\to} &X \\
	\downarrow & & \downarrow f' & & \downarrow f\\
	U_Z & \subset & Z'
	& \stackrel{\pi_Z}{\to} &Z \end{array}
	$$
	such that $\pi_X$ and $\pi_Z$ are birational projective morphisms, $X'$ and $Z'$ are smooth, the inclusions on the left are toroidal embeddings, and such that:
	
	1. $f'$ is toroidal.
	
	
	2. $\Supp \pi_X^*(B +D) \subset X'\backslash U_{X'}$ is a simple normal crossings divisor.
	
	By Proposition \ref{prop-toroidal} we have toroidal birational morphisms $\pi_X':X'' \to X'$, $\pi_Z':Z'' \to Z'$ such that the induced morphism $f'': X'' \to Z''$ is equidimensional toroidal. Moreover, by Lemma \ref{lem-toroidal}, $X''$ admits a small $\Q$-factorialisation $\tau: \overline{X''} \to X''$.
	
	If we write $K_{\overline{X''}}+ \overline{B''}=(\pi_X \circ \pi_X' \circ \tau)^* (K_X+B) +\overline{E}$ where $\overline{B''},\overline{E}$ have no common components, then it suffices to prove that $(\overline{X''}/Z,\overline{B''})$ has a good minimal model. Let 
	$$\overline{D}=(\pi_X \circ \pi_X' \circ \tau)^* D =(\pi_X \circ \pi_X' \circ \tau)^* D_h+(\pi_X \circ \pi_X' \circ \tau)^* D_v=\overline{D}_h+\overline{D}_v.$$
	 We have $K_{\overline{X''}}+ \overline{B''} \sim_\R \overline{D}_h+\overline{D}_v + \overline{E}$. Since $f''$ is equidimensional and $\tau$ is small, every component of $\overline{D}_v$ is mapped onto a prime divisor on $Z''$. It follows that we can write $\overline{D}_v=(f''\circ \tau)^*G + \overline{F}$ where $G \ge 0$ and $\overline{F}$ is very exceptional$/Z''$.
	
	Now we run an LMMP$/Z''$ on $K_{\overline{X''}}+ \overline{B''}$ with scaling of an ample divisor. Because $\overline{D}_h|_{X_\eta} =N_\sigma(\overline{D}_h|_{X_\eta})$, we deduce $ \overline{E}+\overline{F} +\overline{D}_h \le N_\sigma (K_{\overline{X''}}+ \overline{B''}/Z)$. Hence the LMMP terminates with a minimal model $(Y,B_Y)$ on which $K_Y+B_Y \sim_\R 0/Z''$ thanks to [\ref{BH-I}]. Since $(Y,B_Y)$ is klt and $Z'' \to Z$ is birational, by a canonical bundle formula and a standard argument (see proof of Proposition \ref{klt-g-model}) we deduce that $(Y/Z,B_Y)$ has a good minimal model which in turn implies that $(X/Z, B)$ has a good minimal model.
\end{proof}

\begin{proof}[Proof of Theorem \ref{lc-flips-1}]
The argument is similar to [\ref{B-lc-flips}]. By replacing $(X,B)$ with a dlt blow-up, we can assume that $(X,B)$ is $\Q$-factorial dlt. If $(X,B)$ is klt, then $(K_X+B)|_{X_\eta}\sim_\R -A|_{X_\eta} $ is either non-pseudo-effective or $A|_{X_\eta}=0$ where $X_\eta$ is the generic fibre of $f$. Therefore either $(X,B)$ has a Mori fibre space or a good minimal model thanks to Lemma \ref{klt-m-g-model}. From now on we assume that $\lfloor B \rfloor \neq 0$.  

\textbf{The vertical case.} 
We first prove the case when every lc centre of $(X/Z,B)$ is vertical$/Z$. Assume Theorem \ref{lc-flips-1} holds for $\dim \leq n-1$. If $K_X+B$ is not pesudo-effective$/Z$, then $(X/Z,B)$ has a Mori fibre space by [\ref{BCHM}]. If $K_X+B$ is pseudo-effective$/Z$, then $(K_X+B)|_{X_\eta}\sim_\R 0 $. Let $P:= \lfloor B\rfloor$ and pick a decreasing sequence of sufficiently small positive real numbers $\epsilon_1 >\epsilon_2 > \cdots >\epsilon_j > \cdots$ such that $\lim_{j \rightarrow \infty} \epsilon_j =0$. We can run an LMMP$/Z$ on $K_X+B-\epsilon_j P$ which terminates with a good minimal model since $P$ is vertical$/Z$ and $(X/Z,B-\epsilon_j P)$ is klt. As we discussed in Section \ref{preliminaries}: a special LMMP, we obtained a sequence of birational maps
$$
(X_2/Z,B_2) \rightarrow \cdots \rightarrow (X_j/Z,B_j) \rightarrow \cdots
$$ 
which can be embedded into a sequence of LMMP$/Z$ on $K_{X_2}+B_2$ with scaling of $t_2C_2$. It suffices to prove that the LMMP above terminates near $\lfloor B_j \rfloor$ for $j \gg 0$. Let $A_j$ be the birational transform of $A$ on $X_j$. Then, $K_{X_j}+B_j+ A_j \sim_\R 0/Z$ still holds for every $j$. As we discussed in Section \ref{preliminaries}: lifting a sequence of log flips with scaling and by Lemma \ref{lc-m-g-model-2}, we can lift the sequence of birational maps so that after replacing we can assume $(X_j/Z,B_j)$ is dlt. After reindexing we assume the LMMP$/Z$ on $K_{X_2}+B_2$ with scaling which contains the sequence does not contract any lc centre. Let $S_2$ be a component of $\lfloor B_2 \rfloor$, $S_j$ be the birational transform of $S_2$ on $X_j$ and let $T$ be the normalisation of the image
of $S_2$ in $Z$. As we discussed in Section \ref{preliminaries}: on special termination, the sequence 
$$
(S_2/T,B_{S_2}) \rightarrow \cdots \rightarrow (S_j/T,B_{S_j}) \rightarrow \cdots
$$
can be embedded into a sequence of LMMP$/T$ on $K_{S_2}+B_{S_2}$ with scaling of $t_2 C_{S_2}$ where $K_{S_2}+B_{S_2}=(K_{X_2}+B_2)|_{S_2}$ and $C_{S_2}= C_2|_{S_2}$. By an inductive assumption, this LMMP$/T$ terminates since $K_{S_2}+B_{S_2}+A_{S_2} \sim_\R 0/T$, which in turn implies that the LMMP$/Z$ above terminates with a good minimal model$/Z$ thanks to Theorem \ref{lc-flips-3}. Hence $(X/Z,B)$ has a good minimal model by [\ref{B-lc-flips}, Corollary 3.7 and Remark 2.7].

\textbf{The horizontal case.} Now we treat the case when there is an lc centre $S \subset \lfloor B \rfloor$ horizontal$/Z$. For the same reasoning we assume that $(K_X+B)|_{X_\eta}\sim_\R 0 $. Run an LMMP$/Z$ on $K_X+B -\epsilon \lfloor B \rfloor$ for some $\epsilon \ll 1$ which terminates with a Mori fibre space $g:(Y/Z,B_Y) \rightarrow T$. By Theorem \ref{ACC-2} $(Y,B_Y)$ is lc and $K_Y+B_Y \sim_\R 0/T$ provided that $\epsilon$ is sufficiently small. Now let $W$ be a common log resolution of $(X,B)$ and $(Y,B_Y)$, and let $(W,B_W)$ be a log smooth model of $(X,B)$.
$$
\xymatrix{
	& W \ar[dl]_{p} \ar[dr]^{q} \ar@{-->}[r]_{} & W'\ar[d]^{q'} \ar@{-->}[r]_{} & W'' \ar[dd]_{}\\
	X \ar[dd]_{f}\ar@{-->}[rr]_{} &&  Y\ar[d]^{g} &\\
	&& T \ar[dll]_{}   & T'  \ar[l]_{}\\
	Z &} 
$$
Run an LMMP$/Y$ on $K_W+B_W$. Since $q$ is birational, by Lemma \ref{lem-exc-3}, we reach a model $(W',B_{W'})$ after finitely many steps such that $K_{W'}+B_{W'}+G = q'^\ast(K_Y+B_Y)$ for some exceptional$/Y$ divisor $G \ge 0$. Moreover, by Theorem \ref{lc-flips-1} in the vertical case we assume that $K_{W'}+B_{W'}$ is semi-ample$/Y$. 

We use the same argument from the proof of [\ref{B-lc-flips}, Lemma 6.7] to prove that $(W'/T,B_{W'})$ has a good minimal model $W'' \rightarrow T'$. By replacing $(W'/Y,B_{W'})$ with its lc model, we can assume that $G$ contains all exceptional$/Y$ divisors. By adding a small multiple of $G$ to $B_{W'}$ we assume that $\mathrm{Supp} G \subseteq \mathrm{Supp} B_{W'}$ and hence
$$
\mathrm{Supp} q'^\ast \lfloor B_Y\rfloor \subseteq (\mathrm{Supp} G \bigcup \mathrm{Supp} \lfloor B_{W'} \rfloor ) \subseteq \mathrm{Supp}B_{W'}
$$
By replacing $(W'/T, B_{W'})$ with a dlt blow-up, we can again assume that
$(W'/T, B_{W'})$ is $\Q$-factorial dlt. Since $\lfloor B_Y \rfloor$ is ample$/T$, for a sufficiently small number $\tau$, we write
$$
K_{W'}+B_{W'} =K_{W'}+B_{W'} -\tau q'^\ast \lfloor B_Y\rfloor +\tau q'^\ast \lfloor B_Y\rfloor \sim_\R K_{W'}+\Delta_{W'} /T
$$
such that $(W'/T,\Delta_{W'})$ is klt. By Lemma \ref{lc-m-g-model-2} $(W'/T,\Delta_{W'})$ has a good minimal model which in turn implies that $(W'/T,B_{W'})$ has a good minimal model $(W''/T,B_{W''}) \to T'$.

Let $A_{W'}= G+q'^\ast A_Y$. We have $K_{W'}+B_{W'}+A_{W'} \sim_\R 0/Z$. Let $A_{W''}$ be the birational transforms of $A_{W'}$ on $W''$. Then, $K_{W''}+B_{W''}+A_{W''} \sim_\R 0/Z$ and hence $K_{S_{W''}}+B_{S_{W''}}+A_{S_{W''}} :=(K_{W''}+B_{W''}+A_{W''})|_{S_{W''}} \sim_\R 0/Z$. By adjunction formula $(S_{W''},B_{S_{W''}}+A_{S_{W''}})$ is lc. By an inductive assumption, $(S_{W''}/Z,B_{S_{W''}})$ has a good minimal model which in turn implies that $(W''/Z,B_{W''})$ has a good minimal model by Lemma \ref{lc-m-g-model}. It follows that $(X/Z,B)$ has a good minimal model by the construction of $(W''/Z,B_{W''})$.
\end{proof}

\begin{proof}[Proof of Theorem \ref{lc-flips-2}]
The argument is similar as above.	
\end{proof}

\vspace{0.3cm}
\section{Klt pairs and the horizontal case}\label{klt-pairs}

\begin{prop}\label{klt-g-model}
Let $(X,\Delta=A+B)$, $f$ and $Z$ be as in Theorem~\ref{t-main-1}. Assume that $(X,\Delta)$ is klt. Then, either $(X,\Delta)$ has a good minimal model or a Mori fibre space.
\end{prop}

\begin{proof}
Replacing $(X,\Delta)$ with a dlt blow-up and $f$ with its Stein factorization, we can assume that $f$ is a contraction and $(X,\Delta)$ is $\Q$-factorial dlt. If $K_X+\Delta$ is not pseudo-effective, then it has a Mori fibre space thanks to [\ref{BCHM}]. Hence we assume that $K_X+\Delta$ is pseudo-effective. Applying a canonical bundle formula [\ref{FG3}, Theorem 3.1] there exists a boundary $B_Z$ such that $(Z,B_Z)$ is klt and $K_X+B \sim_\R f^\ast (K_Z+B_Z)$. By [\ref{BCHM}], $(Z,\Delta_Z=A_Z+B_Z)$ has a good minimal model $g:(Z',\Delta_{Z'} )\rightarrow T$. \\

We show that $K_X+\Delta$ birationally has a Nakayama-Zariski decomposition with semi-ample positive part which in turn implies the existence of good log minimal model of $(X,\Delta)$ by [\ref{BH-I}]. To this end, let $Z \xleftarrow{p_Z} \widetilde{Z} \xrightarrow{q_Z} Z'$ be a common resolution. In addition, we let $p:\widetilde{X} \longrightarrow X$ be a resolution such that the rational map $\widetilde{f}: \widetilde{X} \dashrightarrow \widetilde{Z}$ is a morphism. Now we have the commutative diagram as follows.
$$
\xymatrix{
	X \ar[d]_{f} &  \widetilde{X}\ar[d]^{\widetilde{f}}\ar[l]_{p} &\\
	Z  &  \widetilde{Z}\ar[l]_{p_Z}\ar[r]^{q_Z} & Z' \ar[d]^{g}     \\
	& & T } 
$$
 Consider the Nakayama-Zariski decomposition $p_Z^\ast (K_Z+\Delta_Z) = P+ N$ and put $G=p^\ast(K_X+B)- \widetilde{f}^\ast (P+N) \sim_\R 0$. It follows that $\widetilde{f}^\ast N \leq N_\sigma (p^\ast(K_X+\Delta))$ because $N$ is exceptional$/Z'$ and hence $\widetilde{f}^\ast N$ is very exceptional$/Z'$ by [\ref{Gongyo-Lehmann}, Lemma 2.16]. On the other hand, $\widetilde{f}^\ast P +G \leq P_\sigma (p^\ast(K_X+\Delta))$ because $P$ is semi-ample which forces that $ P_\sigma (p^\ast(K_X+\Delta))=\widetilde{f}^\ast P $ is semi-ample.
\end{proof}

\begin{thm}\label{horizontal}
Let $(X,\Delta=A+B)$, $f$ and $Z$ be as in Theorem~\ref{t-main-1}. Assume that Theorem~\ref{t-main-1} holds in dimension $n-1$ and that there exists an lc centre of $(X,B)$ which is horizontal$/Z$. Then either $(X,\Delta)$ has a good minimal model or a Mori fibre space.
\end{thm}

\begin{proof}
Replacing $(X,\Delta)$ with a dlt blow-up and $f$ with its Stein factorization, we can assume that $f$ is a contraction and $(X,\Delta)$ is $\Q$-factorial dlt. If $K_X+\Delta$ is not pseudo-effective, then it has a Mori fibre space thanks to [\ref{BCHM}]. Hence we assume that $K_X+\Delta$ is pseudo-effective. By assumption there is an irreducible component $S \subseteq \lfloor B \rfloor$ which is horizontal$/Z$. Let $S \rightarrow Z' \rightarrow Z$ be the Stein factorization, $D_Z$ be a divisor defined by $f^\ast D_Z = K_X +B$, and let $D_Z'$, $A_Z'$ be the pull-backs of $D_Z$, $A_Z$ on $Z'$ respectively. By the inductive assumption and Proposition \ref{klt-g-model}, $(S,\Delta_S)$ has a good minimal model where $K_S+\Delta_S=(K_X+\Delta)|_S$ which in turn implies that $D_Z+A_Z$ is abundant thanks to [\ref{Nakayama}, Chapter II Lemma 3.11 and Chapter V 2.7 Proposition (4)]. In particular, $(S,\Delta_S)$ has the lc model which is the ample model of $D_Z'+A_Z'$, and $K_X+\Delta$ is abundant since $D_Z+A_Z$ is abundant.

Now we prove that $D_Z+A_Z$ has the ample model which is also the lc model of $(X,\Delta)$. To this end, we show that $D_Z+A_Z$ birationally has a Nakayama-Zariski decomposition with semi-ample positive part. Let $\overline{Z}'$ be a resolution of $Z'$ such that $P_\sigma(A_{\overline{Z}'}+D_{\overline{Z}'})$ is semi-ample where $A_{\overline{Z}'}$ and $D_{\overline{Z}'}$ are the pull-backs of $A_{Z'}$ and $D_{Z'}$ on $\overline{Z}'$ respectively. By [\ref{Nakayama}, Chapter III 5.18 Corollary] there is a birational morphism $\overline{Z} \rightarrow Z$ such that $P_\sigma(A_{\overline{Z}}+D_{\overline{Z}})$ is semi-ample where $A_{\overline{Z}}$ and $D_{\overline{Z}}$ are the pull-backs of $A_Z$ and $D_Z$ respectively. Let $T''$ be the ample model of $P_\sigma(A_{\overline{Z}}+D_{\overline{Z}})$. By a similar argument in the proof of Proposition~\ref{klt-g-model}, we deduce that $T''$ is the lc model of $(X,\Delta)$, which in turn implies the conclusion.
\end{proof}

The theorems above has settled the klt case and the horizontal case. As we are about to deal with the vertical case, we prove a slightly generalised result of Proposition \ref{klt-g-model} which admits a more flexible assumption on the relative triviality.

\begin{cor}\label{klt-lc-m}
Let $(X,\Delta=A+B)$ be a projective lc pair of dimension $n$ where $A$, $B \ge 0$, $f:(X,\Delta) \rightarrow Z$ be a surjective morphism to a normal variety such that $K_X+B \sim_{\mathbb{R}} 0/U$ for some open subset $U \subseteq Z$, and let $A\sim_\R f^{\ast} A_Z$ be the pull-back of an ample divisor $A_Z$. Assume that $U$ contains the image of the generic points of all lc centres of $(X,B)$. If either $(X,B)$ is klt or Theorem~\ref{t-main-1} holds in dimension $n$, then either $(X,A+B)$ has a good minimal model or a Mori fibre space.
\end{cor}

\begin{proof}
Replacing $(X,\Delta)$ with a dlt blow-up and $f$ with its Stein factorization, we can assume that $f$ is a contraction and $(X,\Delta)$ is $\Q$-factorial dlt. If $K_X+\Delta$ is not pseudo-effective, then it has a Mori fibre space thanks to [\ref{BCHM}]. Hence we assume that $K_X+\Delta$ is pseudo-effective. Replacing $A_Z$ with a general member of $|A_Z|_\R$, we can assume that $\lfloor A+B \rfloor=\lfloor B \rfloor$. Run an LMMP$/Z$ on $K_X+B$ which terminates with a good minimal model $g:(X'/Z,B') \rightarrow T$ thanks to Theorem \ref{lc-flips-2}. Since $K_X+B\sim_\R 0/U$, $\dim T=\dim Z$ and $h:T \rightarrow Z$ is a birational morphism.

It is enough to prove that $(X',A'+B')$ has a good minimal model where $A'$ is the pull-back of $A_Z$. To this end, observe that the birational map $X \dashrightarrow X'$ does not intersect with $f^{-1}U$ and so, the birational morphism induces an isomorphism on the inverse image of $U$.
$$
\xymatrix{
	X \ar[dd]_{f}\ar@{-->}[r]_{} &  X'\ar[d]^{g} &\\
     & T \ar[dl]^{h}   &  \\
	Z &} 
$$
Moreover, if we pick a sufficiently small number $\epsilon>0$, then
\begin{align*}
K_{X'}+A'+B' &\sim_\R (1-\epsilon)(K_{X'}+B)+(A'+\epsilon(K_{X'}+B')) \\
 &\sim_\R (1-\epsilon)(K_{X'}+B+\frac{1}{1-\epsilon}(A'+\epsilon(K_{X'}+B')))
\end{align*}
where $\frac{1}{1-\epsilon}(A'+\epsilon(K_{X'}+B'))$ is $\R$-linear to the pull-back of some ample divisor on $T$. By the assumption that either $(X,B)$ is klt or Theorem~\ref{t-main-1} holds for dimension $n$, the conclusion follows from Lemma \ref{klt-g-model}.
\end{proof}

\vspace{0.3cm}
\section{The vertical case}\label{vertical-case}

Setup $(\ast)$: Replacing $(X,\Delta=A+B)$ with a dlt blow-up and $f$ with its Stein factorization, we can assume that $f$ is a contraction and $(X,\Delta)$ is $\Q$-factorial dlt. If $K_X+\Delta$ is not pseudo-effective, then it has a Mori fibre space thanks to [\ref{BCHM}]. Hence we assume that $K_X+\Delta$ is pseudo-effective. Replacing $A_Z$ with a general member of $|A_Z|_\R$, we can assume that $\lfloor A+B \rfloor=\lfloor B \rfloor$. Suppose that $\lfloor B\rfloor$ is vertical$/Z$ and is denoted by $P$.

\begin{defn}
Let $X$ be a normal projective variety and $D$ be a pseudo-effective divisor on $X$. We define the non-nef locus of $D$ to be the set
$$
\mathrm{NNF}(D):= \{c_X(v)| \sigma_v(\|D\|)>0 \}
$$
where $v$ runs over all discrete valuations of $X$, $\sigma_v$ denotes the numerical asymptotic vanishing order and $c_X$ denotes the centre of the valuation.
\end{defn}

\begin{lem}\label{NNF}
Let $f:Y\longrightarrow X$ be a surjective morphism of normal projective varieties with connected fibres, and $D$ be a pseudo-effective divisor on $X$. We have the following equation on non-nef loci
$$
\mathrm{NNF}(f^\ast D) = f^{-1} \mathrm{NNF}(D).
$$
\end{lem}

\begin{proof}
This follows directly from [\ref{Nakayama}, Chapter III, 5.15 Lemma].
\end{proof}

\begin{rem}
It is already known that the non-nef locus coincides with the restricted base locus on a smooth variety (or with klt type singularities). Readers can consult [\ref{C-B}] for details. An analogue in positive characteristic is proved in [\ref{Mustata}]. However, it remains an open question in the general case.
\end{rem}

\begin{lem}\label{vertical-lem'}
	Let $(X,\Delta=A+B)$, $f$ and $P$ be as in Setup $(\ast)$. Assume that $K_X+\Delta$ is pseudo-effective. Let $H$ be an ample divisor on $X$, $\delta>0$ be a sufficiently small number and let $(W,\Delta_W+\delta H_W)$ be a good minimal model of $(X,\Delta+\delta H)$. Assume further that Theorem~\ref{t-main-1} holds for dimension $\leq n-1$. Then, $K_W+\Delta_W$ is semi-ample on every lc centre, that is, $(K_W+\Delta_W)|_S$ is semi-ample for every lc centre $S$ of $(W,\Delta_W)$.
\end{lem}

\begin{proof}
By replacing $H$ with a large multiple, we can assume that $K_X+\Delta+H$ is ample. Again by replacing $H$ with a general member of $|H|_\R$, we can assume that $(X,\Delta+H)$ is dlt. Now run an LMMP on $K_X+\Delta +\delta H$ with scaling of $H$ which terminates with a good minimal model by [\ref{BCHM}]. We assume this LMMP ends with $(W,\Delta_W+\delta H_W)$. Because $\mathrm{NNF}(K_X+\Delta+ \delta H) \subseteq \mathrm{NNF}(K_X+\Delta)$ is a closed subset, $f (\mathrm{NNF}(K_X+\Delta+ \delta H)) \subseteq f (\mathrm{NNF}(K_X+\Delta))$. By Lemma~\ref{NNF}, $f (\mathrm{NNF}(K_X+\Delta+ \delta H))$ is a proper closed subset. If $\delta$ is sufficiently small, then $\mathrm{NNF}(K_X+\Delta+ \delta H)$ contains every lc centre which is contained in $\mathrm{NNF}(K_X+\Delta)$. If we choose an open subset $V \subseteq Z\setminus f (\mathrm{NNF}(K_X+\Delta+ \delta H))$ of $Z$ and denote its inverse image on $X$ by $U$, then the birational map $X \dashrightarrow W$ is an isomorphism on $U$. We assume that $U_W$ contains the generic points of all lc centres of $(W,\Delta_W)$ where $U_W$ is the open subset of $W$ which is isomorphic to $U$.
	
If $K_W+\Delta_W$ is not nef on some lc centre, then we again run an LMMP on $K_W+\Delta_W$ with scaling of $\delta H_W$ where $H_W$ is the birational transform of $H$. 
$$
(W,\Delta_W)=(W_1,\Delta_{W_1})\dashrightarrow \dots \dashrightarrow (W_k,\Delta_{W_k})\dashrightarrow(W_{k+1},\Delta_{W_{k+1}}) \dashrightarrow \dots 
$$
We claim that the LMMP above terminates near $\lfloor \Delta_{W_k} \rfloor$ for $k \gg 0$. Moreover, $K_{W_k}+\Delta_{W_k}$ is semi-ample on every lc centre. We therefore conclude the lemma by replacing $\delta$ with some smaller positive number.
	
To this end, let $S_W$ be an lc centre of $(W,\Delta_W)$, $S$ be the corresponding lc centre of $(X,\Delta)$ and $\widetilde{S_{W}}$ be a common log resolution of $(S,\Delta_S)$ and $(S_W,\Delta_{S_W})$. We pick a divisor $\Delta_{\widetilde{S_{W}}}$ so that $(\widetilde{S_{W}}, \Delta_{\widetilde{S_{W}}})$ is a log smooth model of $(S_W,\Delta_{S_W})$. Replacing $W$ with $W_k$ for $k \gg 0$ and by induction we can assume that the original LMMP induces an LMMP on $K_{S_W}+\Delta_{S_W}$ with scaling of $\delta H_{S_W}$ where $H_{S_W}= H_W|_{S_W}$. Consider the following commutative diagram
	$$
	\xymatrix{
		& \widetilde{S_{W}} \ar[dl]_{p} \ar[dr]^{q}  &\\
		S  \ar[d]_{f}\ar@{-->}[rr]_{} & & S_W \\
		T } 
	$$
	where $T$ is defined by the Stein factorization $S \rightarrow T\rightarrow Z$.
	
	We denote the restriction of $U$ to $S$ by $U_S$, the restriction of $U_W$ to $S_W$ by $U_{S_W}$ and the restriction of $V$ to $T$ by $V_T$. Then, $U_S$ is isomorphic to $U_{S_W}$, and hence we assume that $U_{\widetilde{S_{W}}}$ is also isomorphic to $U_S$ where $U_{\widetilde{S_{W}}}$ is the inverse image of $U_S$. It follows that $K_{\widetilde{S_{W}}} +\Delta_{\widetilde{S_{W}}} \sim_\R 0/V_T$. By the construction of $W$ we have $\pi^\ast A \leq \pi_W^\ast A_W $ for any common resolution $X \xleftarrow{\pi} \widetilde{W} \xrightarrow{\pi_W} W$ where $A_W$ is the birational transform of $A$. We therefore deduce that $p^\ast A_S \leq q^\ast A_{S_W}$ where $A_S=A|_S$ and $A_{S_W}=A_W|_{S_W}$. It is easy to check that $\Delta_{\widetilde{S_{W}}} \geq \epsilon q^\ast A_{S_W}$ for some sufficiently small number $\epsilon >0$ and hence $\Delta_{\widetilde{S_{W}}} \geq \epsilon p^\ast A_S$. Then, $(\widetilde{S_{W}},\Delta_{\widetilde{S_{W}}})$ has a good minimal model thanks to Corollary \ref{klt-lc-m} and hence $(S_W,\Delta_{S_W})$ has a good minimal which in turn implies that the induced LMMP on $K_{S_W}+\Delta_{S_W}$ terminates and $K_{S_{W_k}}+\Delta_{S_{W_k}}$ is semi-ample for $k \gg 0$.
\end{proof}

\subsection{The pseudo-effective case}
We treat the case when $K_X+\Delta -\epsilon P$ is pseudo-effective for every sufficiently small number $\epsilon>0$.

\begin{lem}\label{vertical-lem}
Let $(X,\Delta=A+B)$, $f$ and $P$ be as in Setup $(\ast)$. Assume that $K_X+\Delta -\epsilon P$ is pseudo-effective for all sufficiently small numbers $\epsilon>0$. Then, $(X,\Delta)$ and $P$ induces a sequence of log birational models $(X_j,\Delta_j)$ of $(X,\Delta)$ together with a sequence of birational maps  
 $$
 (X_2,\Delta_2) \dashrightarrow (X_3,\Delta_3) \dashrightarrow\cdots \dashrightarrow (X_j,\Delta_j) \dashrightarrow \cdots
 $$ 
for $j \ge 2$ such that 
 
$\bullet$ the sequence can be embedded into a sequence of LMMP on $K_{X_2}+B_2$ with scaling of some divisor $t_2 C_2$, and $\lim\limits_{j\rightarrow \infty} t_j=0$ where $t_j$'s are the coefficients appeared in the LMMP with scaling, where $C_2 \sim_\R K_{X_2}+\Delta_2 -\epsilon P_2$ for some small number $\epsilon>0$, a divisor $P_2 \geq 0$ and $\mathrm{Supp} P_2 = \lfloor \Delta_2 \rfloor$;

$\bullet$ the LMMP above consists of only log flips which does not contract any lc centre;
 
$\bullet$  for every $j$, $(X_j,\Delta_j=A_j+B_j)$ is $\Q$-factorial dlt where $A_j \geq 0$, $\lfloor A_j+B_j\rfloor =\lfloor B_j \rfloor$ and for any common resolution $X \xleftarrow{\pi}\overline{X_j} \xrightarrow{\rho} X_j$, we have $\pi^\ast A \leq \rho^\ast A_j$;

$\bullet$ for every $j$, there exist a nonempty open subset $V_j \subset Z$, $U_j'= f^{-1}V_j \subset X$ and an open subset $U_j \subset X_j$ such that, there is a common resolution $U_j' \xleftarrow{p_j} \overline{U_j} \xrightarrow{q_j} U_j$ of $U_j$ and $U_j'$ with $p_j$, $q_j$ surjective, and
$$
p_j^\ast ((K_X+\Delta)|_{U_j'})=q_j^\ast((K_{X_j}+\Delta_j)|_{U_j});
$$ 

$\bullet$ there exists a divisor $C \geq 0$ on $X$ such that for every $j \ge 2$ and any prime divisor $\Gamma$ over $X$, we have the equation
 \begin{align}\label{eq1}
 a(\Gamma,X,\Delta+t_jC) +\sigma_\Gamma(K_X+\Delta +t_j C) =a(\Gamma,X_j,\Delta_j+t_j C_j).
 \end{align}
From Equation (\ref{eq1}), one deduces that if $\Gamma$ is a prime divisor over $X$ with $a(\Gamma,X_j,\Delta_j)=0$, then $a(\Gamma,X,\Delta)=0$ and $\sigma_\Gamma(K_X+\Delta)=0$;

$\bullet$ there is a $\Q$-factorial dlt blow-up $h:(X',\Delta') \to (X,\Delta)$ such that the birational contraction $X' \dashrightarrow X_j$ is an isomorphism at the generic point of every lc centre of $(X_j,\Delta_j)$. 

In particular, if the LMMP above terminates, then it ends with a weak lc model $(Y,\Delta_Y)$ of $(X,\Delta)$.
\end{lem}

\begin{proof}
Pick a sequence of sufficiently small positive numbers $\epsilon_1>\epsilon_2> \ldots >\epsilon_j> \ldots$ such that $\lim_{j \rightarrow \infty} \epsilon_j =0$, and for each $j$ we run an LMMP on $K_X+\Delta -\epsilon_j P$ which terminates with a good minimal model $(X_j,\Delta_j-\epsilon_j P_j=A_j+B_j -\epsilon_j P_j)$ by Corollary~\ref{klt-lc-m}. We denote by $\phi_j: X \dashrightarrow X_j$ the birational maps. As we pointed out in Section \ref{preliminaries}: a special LMMP, we assume that all $X_j$'s are isomorphic in codimension one and all $(X_j,\Delta_j)$ are $\Q$-factorial lc by ACC for log canonical thresholds Theorem \ref{ACC-1}. Let $C_1 \sim_\mathbb{R} K_{X_1}+\Delta_1 -\epsilon_1 P_1$. Then, we can run an LMMP on $K_{X_2}+\Delta_2$ with scaling of $t_2 C_2$ where $t_2=\frac{\epsilon_2}{\epsilon_1 -\epsilon_2}$ such that for each $j$ the pair $(X_j,\Delta_j)$ appears in this process. 
$$
(X_2,\Delta_2) \dashrightarrow (X_3,\Delta_3) \dashrightarrow \cdots \dashrightarrow (X_j,\Delta_j) \dashrightarrow \cdots
$$ 
Note that $K_{X_j}+\Delta_j+t_j C_j$ is semi-ample where $t_j=\frac{\epsilon_j}{\epsilon_1 -\epsilon_j}$. Since $t_j$ approaches zero as we run this program, we assume further the LMMP above does not contract any lc centres of $(X_j,\Delta_j)$ for every $j$ after a possible reindexing.

For every $j$, we prove that there exists a nonempty open subset $V_j \subset Z$ such that there is an open subset $U_j \subset X_j$ containing every lc centre of $(X_j,\Delta_j)$, for any common resolution $U_j' \xleftarrow{p_j} \overline{U_j} \xrightarrow{q_j} U_j$ of $U_j$ and $U_j'=f^{-1}V_j$, and we have
$$
p_j^\ast ((K_X+\Delta)|_{U_j'})=q_j^\ast((K_{X_j}+\Delta_j)|_{U_j});
$$ 
To this end, we claim that the closed subset $\mathrm{NNF}(K_X+\Delta -\epsilon_j P)$ is vertical$/Z$, which implies that the birational map $X \dashrightarrow X_j$ induces an isomorphism on the inverse image of a nonempty open subset $V$ of $Z$ as long as $V$ is outside $f(\mathrm{NNF}(K_X+\Delta -\epsilon_j P))$. In fact, if we run an LMMP$/Z$ on $K_X+\Delta -\epsilon_j P$, then it terminates with a good minimal model $(X''/Z,\Delta''-\epsilon_j P'') \rightarrow Z'$. We get a commutative diagram
$$
\xymatrix{
	X  \ar[d]_{f}\ar@{-->}[rr]_{} & & X''  \ar[d]_{f'}\\
	Z && Z' \ar[ll]_{} } 
$$
where the bottom arrow is birational and $K_{X''}+\Delta'' -\epsilon_j P''\sim_\R 0/Z'$. It turns out that $\mathrm{NNF}(K_X+\Delta -\epsilon_j P)$ is vertical$/Z$.

 As we want to lift the LMMP constructed above to dlt pairs (see Section \ref{preliminaries}: lifting a sequence of log flips with scaling), we study the lc centres of $(X_j,\Delta_j)$ for $j\geq 2$. Let $S_j$ be an lc centre of $(X_j,\Delta_j)$ and the prime divisor $\Gamma$ be a corresponding divisor over $X_j$ with centre at $S_j$. We have that $a(\Gamma,X_k,\Delta_k)=0$ for all $k$ since the LMMP does not contain any lc centre as we assumed before. Let $0 \leq C\sim_\mathbb{R} K_X+\Delta-\epsilon_1 P$ such that $\phi_{1,\ast}C=C_1$. For any prime divisor $\Gamma$, by comparing log discrepancies we have the following equation
 \begin{align*}
 a(\Gamma,X,\Delta+t_jC) +\sigma_\Gamma(K_X+\Delta +t_j C) =a(\Gamma,X_j,\Delta_j+t_j C_j).
 \end{align*}
 The equation forces $a(\Gamma,X,\Delta)=0$ and $\sigma_\Gamma (K_X+\Delta)=0$ since $t_j$ can be chosen arbitrary small. Note that $(X,\Delta+t_jC)$ is not necessarily log canonical but Equation (\ref{eq1}) always holds since $X_j$ is $\Q$-factorial. Therefore the centre of $\Gamma$ on $X$ is an lc centre $S$ with $S \nsubseteq \mathrm{NNF}(K_X+\Delta)$ as we assumed that $(X,\Delta)$ is $\Q$-factorial dlt. Now we lift the LMMP to a sequence of birational maps
  $$
 (X_2',\Delta_2') \dashrightarrow (X_3',\Delta_3') \dashrightarrow \cdots \dashrightarrow  (X'_j,\Delta'_j) \dashrightarrow
 $$
 which can be embedded into a partial LMMP as we discussed in Section \ref{preliminaries}. Because the original LMMP does not contract any lc centre, all $X_j'$ are isomorphic in codimension one. Let $P'_j$ and $C'_j$ be the pull-backs of $P_j$ and $C_j$ respectively. Moreover, let $A_j'$ be the birational transform $\widetilde{A_j}$, and $B_j'$ be $\widetilde{B_j}+E$ where $\widetilde{B_j}$ is the birational transform and $E$ is the reduced exceptional$/X_j$ divisor. Note that $(X'_j,\Delta'_j -\epsilon_j P'_j)$ is not necessarily a log minimal model of $(X,\Delta -\epsilon_j P)$. However, Equation (\ref{eq1}) still holds for $(X,\Delta)$ and $(X_j',\Delta_j')$. It is not hard to verify that the sequence of birational maps previously constructed satisfies all desired properties. Also note that $C_k'$ and $P_k'$ are birational tranforms of $C_j'$ and $P_j'$ respectively for arbitrary $j, k \geq 2$ because the LMMP constructed above does not contain any lc centre.
 
 Finally we show the existence of such dlt blow-up $(X',\Delta')$. Let $X \xleftarrow{p} \overline{X} \xrightarrow{q} X_2$ be a common log resolution, and let $\overline{\Delta}$, $\overline{\Omega}$ be boundaries such that $(\overline{X},\overline{\Delta})$ and $(\overline{X},\overline{\Omega})$ are log smooth models of $(X,\Delta)$ and $(X_2,\Delta_2)$ respectively. We can assume that $(X_2',\Delta_2')$ is a log minimal model of $(\overline{X}/X_2,\overline{\Omega})$. Since $\overline{X} \dashrightarrow X_2'$ is a sequence of steps of $K_{\overline{X}}+\overline{\Omega}$-MMP, it is an isomorphism at the generic point of every lc centre of $(X_2',\Delta_2')$. Let $(X',\Delta')$ be a log minimal model of $(\overline{X}/X,\overline{\Delta})$. We deduce the required property from Equation \ref{eq1}.
 
 By replacing $(X_j,\Delta_j)$ with $(X_j',\Delta_j')$, we conclude the lemma. By comparing log discrepancies, if the LMMP terminates, then it ends with a weak lc model $(Y,\Delta_Y)$ of $(X,\Delta)$. 
\end{proof}

\begin{thm}\label{vertical-psef}
	Let $(X,\Delta=A+B)$, $f$ and $P$ be as in Setup $(\ast)$. Assume that Theorem~\ref{t-main-1} holds for dimension $\leq n-1$ and that $K_X+\Delta -\epsilon P$ is pseudo-effective for any sufficiently small number $\epsilon>0$. Then, $(X,\Delta)$ has a good minimal model.
\end{thm}

\begin{proof}
\emph{Step 1}. 
We construct a sequence of birational maps
$$
(X_2,\Delta_2) \dashrightarrow (X_3,\Delta_3) \dashrightarrow \cdots \dashrightarrow (X_j,\Delta_j) \dashrightarrow \cdots
$$ 
which satisfies the conditions listed in Lemma \ref{vertical-lem}. Replacing $(X,\Delta)$ with a suitable dlt blow-up we can assume the birational contraction $X \dashrightarrow X_j$ is an isomorphism at the generic point of every lc centre of $(X_j,\Delta_j)$. Pick some $j \geq 2$. Now we take a common log resolution $\overline{W_j}$ of $(X,\Delta)$ and $(X_j,\Delta_j)$, and write $K_{\overline{W_j}}+\Delta_{\overline{W_j}} = p_j^\ast (K_X+\Delta)+E_{\overline{W_j}} $ where $p_j:(\overline{W_j},\Delta_{\overline{W_j}}) \rightarrow X$ is a log smooth model of $(X,\Delta)$. It is easy to see that if $\Gamma$ is an irreducible component of $E_{\overline{W_j}}$, then $a(\Gamma,X_j,\Delta_j)>0$. In fact, if $a(\Gamma,X_j,\Delta_j)=0$, then $a(\Gamma,X,\Delta)=0$ and $\sigma_\Gamma(K_X+\Delta)=0$ by Lemma \ref{vertical-lem}. Next we run an LMMP$/X_j$ on $K_{\overline{W_j}}+\Delta_{\overline{W_j}}$ which terminates with a good minimal model $(W_j/X_j,\Delta_{W_j})$ since $X_j$ is $\Q$-factorial thanks to Lemma \ref{lem-exc-3} and Theorem \ref{lc-flips-1}.
 $$
 \xymatrix{
 	& (\overline{W_j}, \Delta_{\overline{W_j}}) \ar[dd]_{\overline{g_j}}\ar[dl]_{p_j}\ar@{-->}[dr]^{q_j} &  \\
 	(X,\Delta)\ar@{-->}[dr]_{\phi_j} & & (W_j,\Delta_{W_j}) \ar[dl]^{g_j}    \\
 	& (X_j,\Delta_j) & } 
 $$
  We write $K_{W_j}+\Delta_{W_j} +G_j =g_j^\ast(K_{X_j}+\Delta_j)$ where $g_j:W_j \rightarrow X_j$ is a morphism, $G_j \geq 0 $ is exceptional$/X_j$. We note that $E_{\overline{W_j}}$ is exceptional$/W_j$. Also note that $G_j$ contains no lc centre of $(W_j,\Delta_{W_j}+G_j)$ in its support. Otherwise there was a prime divisor $\Gamma$ over $X_j$ with $a(\Gamma,X_j,\Delta_j)=0$ and $a(\Gamma,W_j,\Delta_{W_j})>0$. By Lemma \ref{vertical-lem} we have $a(\Gamma,X,\Delta)=0$ and $\sigma_\Gamma(K_X+\Delta)=0$. Therefore we have $a(\Gamma,\overline{W_j},\Delta_{\overline{W_j}})=0$ and $\sigma_\Gamma(K_{\overline{W_j}}+\Delta_{\overline{W_j}})=0$, and it follows that $a(\Gamma,W_j,\Delta_{W_j})=0$ which is a contradiction. Note that by suitably choosing $\overline{W_j}$ we can assume the birational map $X \dashrightarrow W_j$ is a birational contraction. Because the set of exceptional$/X_j$ prime divisors on $X$ is finite. If we suitably choose $j$, then we can assume that the reduced divisor on $X$ which is contracted by $X \dashrightarrow W_j$ coincides with that contracted by $X \dashrightarrow W_k$ for infinitely many $k$. 

\emph{Step 2}.
 Fix some index $j\geq 2$. We now prove that $K_{W_j}+\Delta_{W_j}$ is movable, that is, $N_\sigma(K_{W_j}+\Delta_{W_j})=0$. As we assumed in Step 1, for infinitely many $k$, the birational map $W_j \dashrightarrow W_k$ is an isomorphism in codimension one. For such $k$, if $\Gamma$ is an exceptional$/X_j$ divisor on $W_j$, then from the equation 
 $$
 \mathrm{ord}_\Gamma G_k=a(\Gamma,W_k,\Delta_{W_k}) -a(\Gamma,X_k,\Delta_k) = a(\Gamma,X,\Delta) -a(\Gamma,X_k,\Delta_k)
 $$ and Equation (\ref{eq1}) we have
 \begin{align}\label{eq2}
 \mathrm{ord}_\Gamma G_k + t_k \mathrm{ord}_\Gamma C_k = t_k \mathrm{ord}_\Gamma C - \sigma_\Gamma (K_X+\Delta+t_k C).
 \end{align}
 This equation forces $\mathrm{ord}_\Gamma G_k + t_k \mathrm{ord}_\Gamma C_k$ converges to zero as $k$ tends to the infinity. We note that $N_\sigma (K_{W_j}+\Delta_{W_j} +G_{k,W_j}+t_k (g_k^\ast C_k)_{W_j})=0$ where $G_{k,W_j}$ and $(g_k^\ast C_k)_{W_j}$ are the birational transforms of $G_k$ and $g_k^\ast C_k$ on $W_j$, which in turn implies that $N_\sigma(K_{W_j}+\Delta_{W_j})=0$. 
Therefore the birational contraction $X \dashrightarrow W_j$ contracts precisely $N_\sigma ({K_X+\Delta})$.

\emph{Step 3}.
 Now let $D_j= G_j+ t_j g_j^\ast C_j$. We claim that one can run an LMMP on $K_{W_j}+\Delta_{W_j}$ with scaling of $D_j$ such that $\lambda= \lim\limits_{k \rightarrow \infty} \lambda_k=0$ where $\lambda_k$ are the coefficients appeared in the process of the LMMP. To this end, as we discussed in Section \ref{preliminaries}: decreasing the coefficients appeared in LMMP with scaling, it is enough to prove that $(W_j,\Delta_{W_j} +\lambda D_j)$ has a good minimal model for each $0< \lambda \leq 1$ thanks to [\ref{B-lc-flips}, Theorem 1.9]. We write
\begin{align*}
D_j&= G_j+ t_j g_j^\ast C_j \\
   &\sim_\mathbb{R} G_j + t_j (K_{W_j}+\Delta_{W_j} +G_j- \epsilon g_j^\ast P_j)
\end{align*}
For each $\lambda \leq 1$ we have
\begin{align*}
 & ~  K_{W_j} + \Delta_{W_j} +\lambda D_j  \\
 & \sim_\R K_{W_j}+\Delta_{W_j} +(\lambda-\delta) (G_j +t_j g_j^\ast C_j)+\delta_j G_j +\delta_j t_j ( K_{W_j}+\Delta_{W_j}+G_j -\epsilon g_j^\ast P_j) \\
 &\sim_\mathbb{R} (1+\delta  t_j)(K_{W_j}+\Delta_{W_j} +\frac{\lambda+ \delta t_j}{1+\delta  t_j} G_j + \frac{(\lambda-\delta)t_j}{1+\delta t_j} g_j^\ast C_j - \frac{\delta t_j \epsilon}{1+\delta t_j} g_j^\ast P_j)
\end{align*}
where $\delta \ll 1$. Note that $g_j^\ast P_j$ contains all lc centres of $(W_j,\Delta_{W_j})$ because $P_j$ contains all lc centres of $(X_j,\Delta_j)$ and $G_j$ does not contain any lc centre. Now temporarily replacing $(W_j,\Delta_{W_j})$ with the lc model$/X_j$ and later recovering it, we can assume that $G_j$ contains all components of the reduced exceptional$/X_j$ divisor. If we write 
$$
\Delta_{W_j}^\lambda =\Delta_{W_j} +\frac{\lambda+ \delta t_j}{1+\delta  t_j} G_j + \frac{(\lambda-\delta)t_j}{1+\delta t_j} g_j^\ast C_j - \frac{\delta t_j \epsilon}{1+\delta t_j} g_j^\ast P_j,
$$
then $(W_j,\Delta_{W_j}^\lambda)$ is klt since $\Delta_{W_j}^\lambda \geq 0$, provided that $\delta \ll 1$. So it is enough to prove that $(W_j,\Delta_{W_j}^\lambda)$ has a good minimal model for each $\lambda>0$. To this end, replacing $\overline{W}_j$ we assume $q_j$ is a morphism and we write $K_{\overline{W}_j}+\Delta_{\overline{W}_j}^\lambda=q_j^*(K_{W_j} +\Delta_{W_j}^\lambda) +E^\lambda$ where $\Delta_{W_j}^\lambda,E^\lambda \ge 0$ has no common components. It is obvious that that the log smooth pair $(\overline{W}_j,\Delta_{\overline{W}_j}^\lambda)$ is klt. 
 $$
 \xymatrix{
 	& (\overline{W_j},\Delta_{\overline{W_j}}^\lambda) \ar[dl]_{p_j} \ar[dr]^{q_j}\ar[dd]^{\overline{g_j}}  &\\
 	(X,\Delta)\ar@{-->}[dr]_{\phi_j}\ar[dd]^{f} & & (W_j,\Delta_{W_j}^\lambda) \ar[dl]^{g_j} \\
    & (X_j,\Delta_j) & \\
    Z }
 $$
By Lemma \ref{vertical-lem}, there exist a nonempty open subset $V \subset Z$ and an open subset $U_j \subset X_j$ such that, for any common resolution $U \xleftarrow{p_j'} \overline{U_j} \xrightarrow{q_j'} U_j$ of $U_j$ and $U=f^{-1}V$ with $p_j'$, $q_j'$ surjective, we have
$$
p_j'^\ast ((K_X+\Delta)|_U)=q_j'^\ast((K_{X_j}+\Delta_j)|_{U_j}).
$$
By the construction of $W_j$, possibly after shrinking $V$ we have
\begin{align*}
p_j^\ast ((K_X+\Delta)|_U)+E'_{\overline{W_j}}|_{\overline{U_j}} &=q_j^\ast((K_{W_j}+\Delta_{W_j})|_{g_j^{-1}U_j}) \\
&=(g_j\circ q_j)^\ast ((K_{X_j}+\Delta_j)|_{U_j})- q_j^\ast G_j|_{\overline{U_j}}
\end{align*}
where $0\leq E'_{\overline{W_j}}\leq E_{\overline{W_j}}$ and $\overline{U_j}=p_j^{-1}U$. Hence $E'_{\overline{W_j}}|_{\overline{U_j}}=G_j|_{\overline{U_j}}=0$ which in turn implies that $g_j$ is an isomorphism on the inverse image of $U_j$ and 
$$
p_j^\ast ((K_X+\Delta)|_U)=q_j^\ast((K_{W_j}+\Delta_{W_j})|_{g_j^{-1}U_j});
$$
Since $P$ is vertical$/Z$, one deduces that $(g_j \circ q_j)^\ast P_j|_{(\overline{W_j})_\eta} =0$ on the generic fibre $(\overline{W_j})_\eta$ of $f \circ p_j$, and hence $(g_j \circ q_j)^\ast C_j|_{(\overline{W_j})_\eta} \sim_\R 0$ which in turn implies that 
$$(K_{\overline{W}_j}+ \Delta_{\overline{W_j}}^\lambda)|_{(\overline{W_j})_\eta} \sim_\mathbb{R} E^\lambda|_{(\overline{W_j})_\eta} = N_\sigma((K_{\overline{W}_j}+ \Delta_{\overline{W_j}}^\lambda)|_{(\overline{W_j})_\eta}).
$$ 

Since $\Delta_{W_j} +G_j \geq \iota g_j^\ast A_j $ for some sufficiently small number $\iota >0$, by the construction of $\Delta_{W_j}^\lambda$ and adding a small multiple of exceptional$/W_j$ divisors we have $\Delta_{\overline{W_j}}^\lambda \geq \iota' (g_j\circ q_j)^\ast A_j \geq \iota' p_j^\ast A$ for some sufficiently small number $\iota' \ll \iota $. Then, the klt pair $(\overline{W}_j,\Delta_{\overline{W}_j}^\lambda)$ has a good minimal model by [\ref{BH-I}] and Corollary \ref{klt-lc-m} which in turn implies that $(W_j,\Delta_{W_j}^\lambda)$ has a good minimal model. This completes the argument for our claim.

\emph{Step 4}.
Let $H$ be an ample divisor on $X$ and pick a sufficiently small number $\delta_2>0$. Let $(W,\Delta_W +\delta_2 H_W)$ be a good minimal model of $(X,\Delta+ \delta_2 H)$. By Lemma \ref{vertical-lem'}, $K_W+\Delta_W$ is semi-ample on every lc centre. Because $K_{W_j}+\Delta_{W_j}$ is movable, $W$ is isomorphic to $W_j$ in codimension one. We show that the LMMP on $K_{W_j}+\Delta_{W_j}$ with scaling of $D_j$ terminates near $\lfloor\Delta_{W_k'} \rfloor$ after finitely many steps. 
$$
(W_j,\Delta_{W_j})=(W_j',\Delta_{W_j'}) \dashrightarrow \cdots \dashrightarrow (W_k',\Delta_{W_k'}) \dashrightarrow \cdots
$$
To this end, consider an lc centre $S_{W_j}$ of $(W_j,\Delta_{W_j})$. We write $D=D_j$ and $D_{W_k'}, D_W$ as its birational transforms on $W_k',W$. Because $(W_k',\Delta_{W_k'}+\lambda_k D_{W_k'})$ is a good minimal model of $(W,\Delta_W+\lambda_k D_{W})$, and $(W,\Delta_W+\delta_2 H_W)$ is a good minimal model of $(W_k',\Delta_{W_k'}+\delta_2 H_{W_k'})$, for any common resolution $W_k' \xleftarrow{\psi_k} \widetilde{W_k} \xrightarrow{\phi_k} W$, we have the following inequalities 
$$
\psi_k^\ast (K_{W_k'}+\Delta_{W_k'} +\delta_2 H_{W_k}')  \geq \phi_k^\ast (K_{W} +\Delta_{W} +\delta_2 H_{W}) 
$$ 
and 
$$
\psi_k^\ast (K_{W_k'}+\Delta_{W_k'}+ \lambda_k D_{W_k'})  \leq  \phi_k^\ast (K_W+\Delta_W +\lambda_k D_{W}).
$$ 
As we discussed in Section \ref{preliminaries}: on special termination and by induction, the original LMMP induces an LMMP on $K_{S_{W_k'}}+\Delta_{S_{W_k'}}$ with scaling for $k \gg 0$, where $S_{W_k'}$ is the birational transform of $S_{W_j}$. Since the birational map $W \dashrightarrow W_k'$ is an isomorphism at the generic point of every lc centre, by restriction we obtain that
$$
a(\Gamma,S_{W_k'},\Delta_{S_{W_k'}}+\delta_2 H_{S_{W_k'}}) \leq a(\Gamma,S_W,\Delta_{S_W}+\delta_2 H_{S_W})
$$
and
$$
a(\Gamma,S_{W_k'},\Delta_{S_{W_k'}}+\lambda_k D_{S_{W_k'}}) \geq a(\Gamma,S_W,\Delta_{S_W}+\lambda_k D_{S_W})
$$
for every prime divisor $\Gamma$ over $S_W$. Because $\delta_2$ can be chosen arbitrary small, we have
$$
a(\Gamma,S_{W_k'},\Delta_{S_{W_k'}}) \leq a(\Gamma,S_W,\Delta_{S_W}).
$$ 
It follows that 
$$
0 \leq a(\Gamma,S_W,\Delta_{S_W})- a(\Gamma,S_{W_k'},\Delta_{S_{W_k'}}) \leq \lambda_k (\mathrm{ord}_\Gamma D_{S_W}-\mathrm{ord}_\Gamma D_{S_{W_k'}}). 
$$
As we pointed out earlier, the sequence of birational maps
$$
 \cdots \dashrightarrow (S_{W_k'},\Delta_{S_{W_k'}}) \dashrightarrow (S_{W_{k+1}'},\Delta_{S_{W_{k+1}'}}) \dashrightarrow \cdots
$$
can be embedded into sequence of log flips. If $\Gamma$ is a prime divisor on $S_{W_k'}$, then $a(\Gamma,S_{W_k'},\Delta_{S_{W_k'}})=a(\Gamma,S_{W_l'},\Delta_{S_{W_l'}})$ and $\mathrm{ord}_\Gamma D_{S_{W_k'}}=\mathrm{ord}_\Gamma D_{S_{W_l'}}$ for $l>k$ which in turn implies that $a(\Gamma,S_W,\Delta_{S_W})= a(\Gamma,S_{W_k'},\Delta_{S_{W_k'}})$. If $\Gamma$ is a prime divisor on $S_{W}$, then $\lambda_k (\mathrm{ord}_\Gamma D_{S_W}-\mathrm{ord}_\Gamma D_{S_{W_k'}})$ converges to zero as $k$ tends to the infinity which gives $a(\Gamma,S_W,\Delta_{S_W})- a(\Gamma,S_{W_k'},\Delta_{S_{W_k'}}) \leq \sigma_\Gamma (K_{S_{W_k'}}+\Delta_{S_{W_k'}})$. Therefore by Lemma \ref{weak-lc-model-lem} we deduce that $(S_{W_k'},\Delta_{S_{W_k'}})$ has a good minimal model which in turn implies that the LMMP on $K_{W_j}+\Delta_{W_j}$ with scaling of $D_j$ terminates near $S_{W_k'}$. Moreover, since $K_{S_W}+\Delta_{S_W}$ is semi-ample, $K_{S_{W_k'}}+\Delta_{S_{W_k'}}$ is also semi-ample for $k \gg 0$. We continue this argument by induction on dimension and conclude that the LMMP on $K_{W_j}+\Delta_{W_j}$ with scaling of $D_j$ terminates near $\lfloor \Delta_{W_k'} \rfloor$, and that $K_{W_k'}+\Delta_{W_k'}$ is semi-ample on every lc centre for $k \gg 0$.

\emph{Step 5}.
Let $l,k \gg 0$. If $\widetilde{W_{kl}}$ is a common resolution of $W_k'$ and $W_l$,
$$
\xymatrix{
	& \widetilde{W_{kl}} \ar[dl]_{\rho_k'} \ar[dr]^{\rho_l}  &\\
	W_k'\ar@{-->}[rr]_{}\ar@{-->}[drr]_{g_{kl}'} & & W_l \ar[d]^{g_l} \\
	&& X_l} 
$$
then by Negativity Lemma
$$
\rho_k'^{\ast}(K_{W_k'}+\Delta_{W_k'}+ \lambda_k D_{W_k'}) \leq \rho_l^\ast(K_{W_l}+\Delta_{W_l}+ \lambda_k D_{W_l})
$$ 
where $D_{W_l}$ is the birational transform of $D_{W_k'}$ (by abuse of notations) which is also the birational transform of $D_j$. On the other hand, by Negativity Lemma
$$
\rho_k'^{\ast}((K_{W_k'}+\Delta_{W_k'}+  D_{l,W_k'}) \geq \rho_l^\ast(K_{W_l}+\Delta_{W_l}+ D_l)
$$
where $D_{l,W_k'}$ is the birational transform of $D_l$. Let $S_l$ be an lc centre of $(X_l,\Delta_l)$, and let $S_{W_k'},S_{W_l}$ be corresponding lc centres of $(W_k',\Delta_{W_k'}), (W_l,\Delta_{W_l})$ so that the commutative diagram of birational maps above
induces a commutative diagram of birational maps.
$$
\xymatrix{
    S_{W_k'}\ar@{-->}[rr]_{}\ar@{-->}[drr]_{g_{kl}'} & & S_{W_l} \ar[d]^{g_l} \\
	&& S_l} 
$$
By restriction we obtain the following inequalities
\begin{align}\label{ineq-1}
a(\Gamma, S_{W_k'}, \Delta_{S_{W_k'}}+ \lambda_k D_{S_{W_k'}} ) \geq a(\Gamma, S_{W_l}, \Delta_{S_{W_l}}+  \lambda_k D_{S_{W_l}} )
\end{align}
where $D_{S_{W_l}} = D_{W_l}|_{S_{W_l}}$ and 
\begin{align}\label{ineq-2}
a(\Gamma, S_{W_k'}, \Delta_{S_{W_k'}}+ D_{l,S_{W_k'}} ) \leq a(\Gamma, S_{W_l},\Delta_{S_{W_l}}+ D_{l,S_{W_l}} )
\end{align}
where $D_{l,S_{W_k'}}=D_{l,W_k'}|_{S_{W_k'}}$ for any prime divisor $\Gamma$ over $S$. From Inequality (\ref{ineq-1}) we deduce that
$$
\lambda_k (\mathrm{ord}_\Gamma D_{S_{W_k'}} -\mathrm{ord}_\Gamma  D_{S_{W_l}}) \leq a(\Gamma, S_{W_k'}, \Delta_{S_{W_k'}}) -a(\Gamma, S_{W_l},\Delta_{S_{W_l}}).
$$
Because the LMMP on $K_{W_j}+\Delta_{W_j}$ terminates near $\lfloor \Delta_{W_k'} \rfloor$ (by Step 4), we have $\mathrm{ord}_\Gamma D_{S_{W_k'}}=\mathrm{ord}_\Gamma D_{S_{W_m'}}$ and $a(\Gamma, S_{W_k'}, \Delta_{S_{W_k'}})=a(\Gamma, S_{W_m'}, \Delta_{S_{W_m'}})$ for every integer $m>k$ which in turn implies that 
\begin{align}\label{ineq-3}
a(\Gamma, S_{W_k'}, \Delta_{S_{W_k'}}) \geq a(\Gamma, S_{W_l},\Delta_{S_{W_l}})
\end{align}
for every prime divisor $\Gamma$ over $S$. On the other hand, from Inequality (\ref{ineq-2}) we deduce that 
$$
0 \le a(\Gamma, S_{W_k'}, \Delta_{S_{W_k'}}) -a(\Gamma, S_{W_l},\Delta_{S_{W_l}})
\leq \mathrm{ord}_\Gamma D_{l,S_{W_k'}} - \mathrm{ord}_\Gamma D_{l,S_{W_l}}.
$$
If $\Gamma$ is a prime divisor on $S_{W_k'}$, then by Equation (\ref{eq2}) the coefficients of $D_{l,W_k'}$ converge to zero as $l$ tends to infinity which implies that $\mathrm{ord}_\Gamma D_{l,S_{W_k'}}$ converges to zero as $l$ tends to infinity. In particular, $\mathrm{ord}_\Gamma D_{l,S_{W_l}}$ converges to zero as $l$ tends to infinity.

\emph{Step 6}.
We show that $K_{X_l}+\Delta_l$ is semi-ample on every lc centre of $(X_l,\Delta_l)$ for $l \gg 0$ which completes the whole proof. To this end, let $S_l$ be an lc centre of $(X_l,\Delta_l)$, and consider the following commutative diagram.
$$
\xymatrix{
	S_{W_k'}\ar@{-->}[r]_{}\ar@{-->}[dr]_{g_{kl}'}  & S_{W_l} \ar[d]^{g_l}\ar@{-->}[r]_{} & S_{W_{l+1}} \ar[d]^{g_{l+1}}\ar@{-->}[r]_{} & \cdots \\
	& S_{l} \ar@{-->}[r]_{} & S_{l+1} \ar@{-->}[r]_{} & \cdots}
$$
As we pointed out in the previous step, if $\Gamma$ is a prime divisor on $S_{W_k'}$, then 
\begin{align}\label{ineq-4}
a(\Gamma, S_{W_k'}, \Delta_{S_{W_k'}}) &\le a(\Gamma, S_{W_l},\Delta_{S_{W_l}})+\mathrm{ord}_\Gamma D_{l,S_{W_k'}} - \mathrm{ord}_\Gamma D_{l,S_{W_l}} \\
&= a(\Gamma,S_l,\Delta_{S_l}) + \mathrm{ord}_\Gamma D_{l,S_{W_k'}} - t_l \mathrm{ord}_\Gamma C_{l,S_{W_l}}\notag
\end{align}
where $C_{l,S_{W_l}}=(g_l^\ast C_l)|_{S_{W_l}}= g_l^\ast (C_l|_{S_l})$. Since $\mathrm{ord}_\Gamma D_{l,S_{W_k'}} - t_l \mathrm{ord}_\Gamma C_{l,S_{W_l}}$ converges to zero as $l$ tends to infinity, we have that $a(\Gamma, S_{W_k'}, \Delta_{S_{W_k'}}) - a(\Gamma,S_l,\Delta_{S_l}) \leq \sigma_\Gamma (K_{S_l}+\Delta_{S_l})$. Hence by combining Inequalities (\ref{ineq-3}), (\ref{ineq-4}) and Lemma \ref{weak-lc-model-lem}, $(S_l,\Delta_{S_l})$ has a good minimal model which in turn implies that the LMMP on $K_{X_j}+\Delta_j$ terminates near $S_l$ and $K_{S_l}+\Delta_{S_l}$ is semi-ample for $l \gg 0$. We continue this argument by induction on dimension and conclude that the LMMP on $K_{X_j}+\Delta_j$ terminates near $\lfloor \Delta_l \rfloor$ and $K_{X_l}+\Delta_l$ is semi-ample on every lc centre of $(X_l,\Delta_l)$ for $l \gg 0$. By Lemma \ref{vertical-lem} and Theorem \ref{lc-flips-3} we obtain the conclusion.
\end{proof}

\subsection{The non-pseudo-effective case}

\begin{thm}\label{vertical-non-psef}
Let $(X,\Delta=A+B)$, $f$ and $Z$ be as in Setup $(\ast)$. Assume Theorem~\ref{t-main-1} holds for dimension $n-1$. Assume further that $K_X+\Delta -\epsilon P$ is not pseudo-effective for any sufficiently small number $\epsilon>0$. Then, either $(X,\Delta)$ has a good minimal model or a Mori fibre space.
\end{thm}

\begin{proof}
\emph{Step 1}.
Since $K_X+\Delta -\epsilon P$ is not pseudo-effective for $\epsilon \ll 1$, we run an LMMP on $K_X+\Delta -\epsilon P$ with scaling of an ample divisor which terminates with a Mori fibre space $g: (X',\Delta') \rightarrow Y$ by [\ref{BCHM}]. Because $\epsilon$ can be chosen arbitrary small, we obtain that $(X',\Delta')$ is $\Q$-factorial lc and $K_{X'}+\Delta' \sim_\mathbb{R} 0/Y$ by ACC for numerically trivial pairs Theorem \ref{ACC-2}. Let $W$ be a common log resolution of $(X,\Delta)$ and $(X',\Delta')$, 
$$
\xymatrix{
	& (W,\Delta_W) \ar[dl]_{p} \ar[dr]^{q}  &\\
	(X,\Delta) \ar[d]^{f}\ar@{-->}[rr]_{} & & (X',\Delta') \ar[d]^{g} \\
	Z && Y} 
$$
and write $K_W+\Delta_W = p^\ast (K_X+\Delta) +E$  where $(W,\Delta_W)$ is a log smooth model of $(X,\Delta)$. We can run an LMMP$/Y$ on $K_W+\Delta_W$ which terminates with a good minimal model. To see this, we first run an LMMP$/X'$ on $K_W+\Delta_W$ and by Lemma \ref{lem-exc-3} after finitely many steps we reach a model $W'$ on which $K_{W'}+\Delta_{W'} +G \sim_\mathbb{R} 0/Y$ for some divisor $G\geq 0$. Because $K_X+\Delta$ is pseudo-effective, $K_{W'}+\Delta_{W'}$ is also pseudo-effective and hence $G$ is vertical$/Y$. Then, by Theorem \ref{lc-flips-1} we run an LMMP$/Y$ on $K_{W'}+\Delta_{W'}$ which terminates with a good minimal model, after replacing, which is again denoted by $g': (W',\Delta_{W'}) \rightarrow Y'$. Now consider the following commutative diagram.
 $$
 \xymatrix{
 	& (W,\Delta_W) \ar[d]_{q}\ar[dl]_{p}\ar@{-->}[r]^{\phi} & (W',\Delta_{W'}) \ar[d]_{}  \\
 	(X,\Delta)\ar[d]^{f}\ar@{-->}[r]_{} & (X',\Delta')\ar[d]^{g} & Y' \ar[dl]_{}    \\
 	Z &  Y & } 
 $$ 
Since there exists a component $S'$ of $\lfloor\Delta' \rfloor$ on $X'$ which is horizontal$/Y$, the birational transform $S_{W'}$ is a component of $\lfloor \Delta_{W'} \rfloor$ which is horizontal$/Y$. We denote by $S$, $S_W$ the birational transforms on $X,W$.

\emph{Step 2}.
By Lemma \ref{NNF} and Proof of Lemma \ref{vertical-lem'}, there exists a nonempty open subset $V \subset Z$ such that its inverse image $U$ on $X$ contains the generic points of all lc centres of $(X,\Delta)$ not contained in $\mathrm{NNF}(K_X+\Delta)$. In particular, $U$ contains the generic point of $S$ since $\sigma_{S_W}(K_W+\Delta_W)=\sigma_{S_{W'}}(K_{W'}+\Delta_{W'})=0$.

By replacing $(W,\Delta_W)$ we can assume that the birational map $ \phi :W \dashrightarrow W'$ is a morphism. Possibly after shrinking $V$ we may assume that $\phi ^\ast (K_{W'}+\Delta_{W'}) \sim_\mathbb{R} E' /U$ where $0 \leq E' \leq E$ and hence $\mathrm{ord}_\Gamma E' =a(\Gamma, X,\Delta)-a(\Gamma, W',\Delta_{W'})$ for any component $\Gamma$ of $E$. If $\Gamma$ is a prime divisor on $W'$, then 
$$
a(\Gamma,W',\Delta_{W'})+\mathrm{ord}_\Gamma E'=a(\Gamma, X,\Delta).
$$ 
On the other hand, if $\Gamma$ is a prime divisor on $X$, then
$$
a(\Gamma,W',\Delta_{W'}) \ge a(\Gamma,W,\Delta_{W})= a(\Gamma, X,\Delta).
$$
It follows that
$$
\phi^\ast(K_{W'}+\Delta_{W'}) +N = p^\ast (K_{X}+\Delta) +E' +F
$$
where $N \geq 0$ is exceptional$/W'$, $E'$, $F \geq 0$ is exceptional$/X$ with $N$, $E'$ and $F$ being simple normal crossing and having no common components. In addition, $N$ and $F$ are mapped into $X \backslash U$. We therefore deduce that $E'+F \leq N_\sigma (\phi^\ast(K_{W'}+\Delta_{W'}))$. 

\emph{Step 3}.
Let $K_{S_{W'}}+\Delta_{S_{W'}}= (K_{W'}+\Delta_{W'})|_{S_{W'}}$. We denote by $\phi_S :S_W \rightarrow S_{W'}$ the restriction of $\phi$ on $S_W$. Since $S_{W'}$ is horizontal$/Y'$ and $K_{W'}+\Delta_{W'} \sim_\mathbb{R} 0/Y'$, we have 
\begin{align*}
E'|_{S_W}+F|_{S_W} &\leq N_\sigma (\phi^\ast(K_{W'}+\Delta_{W'}) )|_{S_W}\\
& =N_\sigma (\phi_S^\ast (K_{S_{W'}}+\Delta_{S_{W'}}))
\end{align*}
by a similar argument in the proof of Proposition \ref{klt-g-model}.

If we write $K_{S_W}+\Theta=\phi_S^\ast (K_{S_{W'}}+\Delta_{S_{W'}})+ L$ such that $(S_W,\Theta)$ is a log smooth model of $(S_{W'},\Delta_{S_{W'}})$, then we have
\begin{align}\label{eq3}
K_{S_W} + \Theta= p_S^\ast (K_{S}+\Delta_{S}) +E'|_{S_W} +F|_{S_W} -N|_{S_W} + L.
\end{align}
We denote by $M_{S_W}:= E'|_{S_W} +F|_{S_W}  + L$. Since $M_{S_W} \leq N_\sigma (K_{S_W} + \Theta)$, we can decrease the coefficients of $\Theta$ as well as the coefficients of $M_{S_W}$ so that $\Theta$ and $M_{S_W}$ have no common components. It turns out that, if we write 
$$
p_{S,\ast} \Theta + p_{S,\ast} (N|_{S_W})= \Delta_{S} + p_{S,\ast} M_{S_W},
$$ 
then $p_{S,\ast} M_{S_W} \le p_{S,\ast} (N|_{S_W})$.  
$$
\xymatrix{
	S_W'\ar[ddr]_{f''}\ar[dr]^{p'_S} & & S_W \ar[d]^{q_S}\ar[dl]_{p_S}\ar[r]^{\phi_S} \ar@{-->}[ll]_{\psi_S}  & S_{W'} \ar[d]_{}  \\
	& S\ar[d]^{f_S} \ar@{-->}[r]_{} &S' \ar[d]^{g_{S}} & Y' \ar[dl]_{}    \\
	& T	&  Y & } 
$$ 

\emph{Step 4}.
In this step we prove that $(S_{W'},\Delta_{S_W'})$ has a good minimal model. Now run an LMMP$/S$ on $K_{S_W}+\Theta$ with scaling of an ample divisor. By Equation (\ref{eq3}), we have  
$$
K_{S_W} + \Theta \sim_\R M_{S_W} -N|_{S_W} /S.
$$
Because $M_{S_W}$, $N|_{S_W} \geq 0$ and $p_{S,\ast} (M_{S_W}  -N|_{S_W}) \le 0$, by Lemma \ref{lem-exc-3}, after finitely steps we reach a model $S_W'$ on which 
$$
K_{S_W'} + \Theta' \sim_\R - N'/S
$$
where $0 \leq N' = (N|_{S_W}-M_{S_W})'$, and the divisors $\Theta'$, $ (N|_{S_W}-M_{S_W})'$ are birational transforms of $\Theta$, $ N|_{S_W}-M_{S_W}$ respectively. In addition we have 
$$K_{S_W'} + \Theta' +N' = p_S'^\ast (K_{S}+\Delta_{S}).$$ As we pointed out earlier, $N|_{S_W}$ is mapped into $S \backslash U_{S}$ where $U_{S}$ is the restriction of $U$ to $S$, and hence $N'$ is also mapped into $S \backslash U_{S}$. 

Now we write $f''=f_{S}\circ p_S'$. If we suitably choose $A_Z$, then $A_T$ does not contain $T \backslash V_T$ in its support where $T$ is the normalization of the image of $S$ in $Z$, $V_T$ is the inverse image of $V$ on $T$, and $A_T$ denotes that restriction of $A_Z$. So, we have $\Theta' \geq   f''^\ast A_T$ which in turn implies that $(S_W',\Theta')$ has a good minimal model by combining Theorem \ref{lc-flips-1} and Theorem \ref{t-main-1} in dimension $n-1$. We therefore conclude that $(S_{W'},\Delta_{S_W'})$ has a good minimal model.

\emph{Step 5}.
By a similar argument in the proofs of Proposition~\ref{klt-g-model} and Theorem~\ref{horizontal}, we deduce that $(W',\Delta_{W'})$ has a good minimal model which in turn implies that $(X,\Delta)$ has a good minimal model.
\end{proof}

\begin{proof}[Proof of Theorem \ref{t-main-1}]
We argue by induction so in particular we can assume
that Theorem \ref{t-main-1} holds in dimension $ \le n - 1$. By Theorem \ref{vertical-psef} and Theorem \ref{vertical-non-psef}, Theorem \ref{t-main-1}
holds in dimension $n$ in the vertical case. On the other hand, by Theorem \ref{horizontal}, Theorem \ref{t-main-1} also holds in dimension $n$ in the horizontal case.	
\end{proof}

\begin{cor}\label{t-main-cor}
Let $(X,\Delta)$ be a projective lc pair such that $\Delta \geq A$ where $A \geq 0$ is an ample $\R$-divisor. Then, either $(X,\Delta)$ has a good minimal model or a Mori fibre space. 
\end{cor}

\begin{proof}
Let $(Y,\Delta_Y)$ be a $\Q$-factorial dlt blow-up of $(X,\Delta)$. We denote by $f:Y \rightarrow X$. By definition we have $K_Y+\Delta_Y \sim_\R 0/X$. Since $\Delta \ge A$, we have $\Delta_Y \ge A_Y:=f^\ast A$. Hence the corollary follows from Theorem \ref{t-main-1} immediately.
\end{proof}

\vspace{0.3cm}
\section{Terminations with scaling}\label{term}

As an application we prove the terminations with scaling for lc pairs under a positivity assumption. To this end we study the geography of weak lc models on a lc pair. The geography of models on a klt pair was studied by C. Birkar, P. Cascini, C. D. Hacon, J. McKernan [\ref{BCHM}], Y. Kawamata [\ref{Kaw-3}], [\ref{Kaw-4}], A. Corti, V. Lazi\'c [\ref{Corti-Lazic}], P. Cascini, V. Lazi\'c [\ref{Cascini-Lazic}] etc., and the geography of models on a $\Q$-factorial variety was studied by A. S. Kaloghiros, A. K\"uronya, V. Lazi\'c [\ref{KKL}], etc. 

\begin{lem}\label{lem-geography-1}
Let $X$ be a projective normal variety, and let $\{ \Delta_1$, $\Delta_2$, $\ldots$, $\Delta_r \}$ be a set of $\Q$-Weil divisors such that $(X,\Delta_i)$ is lc for each $1\le i \le r$. Let $\mathcal{P}$ be the rational polytope defined by $\{ \Delta_i \}$. Then, the subset
$$
\mathcal{N}:=\{\Delta \in \mathcal{P}|K_X+\Delta \mathrm{~is~nef.}  \}
$$
is a rational polytope contained in $\mathcal{P}$.
\end{lem}

\begin{proof}
The lemma is obvious when $X$ is $\Q$-factorial klt. We therefore assume that $X$ is not $\Q$-factorial klt. Pick a boundary $\Delta \in  \mathcal{P}$, and let $(Y,\Delta_Y)$ be a dlt blow-up of $(X,\Delta)$ where $\Delta_Y=\widetilde{\Delta}+E$. Here $\widetilde{\Delta}$ denotes the birational transform of $\Delta$ and $E$ is the reduced exceptional$/X$ divisor. If $\Delta$ is an interior point of $\mathcal{P}$, it follows that $(Y,\Delta_Y')$ is a $\Q$-factorial dlt blow-up of $(X,\Delta')$ where $\Delta_Y'=\widetilde{\Delta'}+E$ if $\Delta' \in \mathcal{P}$ is sufficiently close to $\Delta$. Since the question is local, by shrinking $\mathcal{P}$ we can assume that $(Y,\Delta_Y)$ is a dlt blow-up of $(X,\Delta)$ for every $\Delta \in \mathcal{P}$. Because $K_X+\Delta$ is nef if and only if $K_Y+\Delta_Y$ is nef, one easily conclude that $\mathcal{N}$ is a rational polytope.

If $\Delta$ lies on a face of $\mathcal{P}$, then for an interior point $\Delta'$ sufficiently close to $\Delta$, it may happen that $a(\Gamma,X,\Delta')>0$. Now we run an LMMP$/X$ on $K_Y+\widetilde{\Delta'} +E$ which terminates with a $\Q$-factorial dlt blow-up of $(X,\Delta')$. Without confusion we again denote it by $(Y,\Delta_Y')$. By an easy calculation we deduce that there exists a neighborhood $U$ of $\Delta$ on $\mathcal{P}$ such that for any point $\Delta'' \in U$, we have that $(Y,\Delta_Y'')$ is $\Q$-factorial lc. By shrinking $\mathcal{P}$ we can lift every divisor $\Delta'' \in \mathcal{P}$ to $Y$ which in turn implies that $\mathcal{N}$ is a rational polytope.
\end{proof}

The next lemma is essentially contained in [\ref{B-II}, Proposition 3.2(5)]. We borrow the proof below for the reader's convenience.

\begin{lem}[cf.\text{[\ref{B-II}, Proposition 3.2(5)]}]\label{lem-geograpgy}
	Let $X$ be a $\Q$-factorial projective variety, and let $\{ \Delta_i \}$ and $\mathcal{P}$ be as in Lemma \ref{lem-geography-1}. Fix a boundary $\Delta \in \mathcal{P}$ such that $K_X+\Delta$ is nef. Then, there is a real number $\epsilon>0$ depending only on $X$, $\Delta$ and $\mathcal{P}$ satisfying: for any boundary $\Delta' \in \mathcal{P}$, if $\|\Delta'-\Delta\|<\epsilon$, then any sequence of $K_X+\Delta'$-MMP is $K_X+\Delta$-trivial.
\end{lem}
\begin{proof}
	Write $\Delta = \Sigma a_j \Delta_j$ as a convex combination of $\Q$-divisors $\Delta_j$. Let $R$ be an extremal ray of $\overline{NE}(X)$. Pick a rational curve $\Gamma \in R$ such that $(K_X+\Delta_j) \cdot \Gamma \ge -2 \dim X$. Then, there is a real number $\alpha >0$ satisfying 
	$$
	(K_X+\Delta) \cdot \Gamma =\Sigma a_j(K_X+\Delta_j) \cdot \Gamma > \alpha
	$$
	if it is positive.
	
	Take $L$ to be the line which goes through $\Delta'$ and $\Delta$ and let $\Delta''$ be the intersection point of $L$ and the boundary of $\mathcal{P}$, in the direction of $\Delta'$. So, there are nonnegaitve real numbers $r, s$ such that $r + s = 1$	and $\Delta' = r\Delta + s\Delta''$. Given a sequence of $K_X+\Delta'$-MMP, we suppose by induction that $X \dashrightarrow X_i$ is $K_X+\Delta$-trivial. In particular, the map $X_i \dashrightarrow X_{i+1}$ is
	also a step of LMMP on $K_X + \Delta''$ and $(X_i, \Delta''_i)$ is lc where $\Delta''_i$ is the
	birational transform of $\Delta''$. Suppose that there is an extremal ray $R$ on $X_i$ such that $(K_{X_i} + \Delta'_i)\cdot R \le 0$ but $(K_{X_i} + \Delta_i)\cdot R > 0$. Let $\Gamma$ be an extremal curve for $R$. By the previous discussion, $(K_{X_i} + \Delta_i) \cdot \Gamma > \alpha$ and $
	(K_{X_i} + \Delta''_i) \cdot \Gamma \ge -2 \dim X$.	Now
	$$
	(K_{X_i} + \Delta'_i) \cdot \Gamma = r(K_{X_i} + \Delta_i)  \cdot \Gamma  + s(K_{X_i} + \Delta''_i) \cdot \Gamma > r\alpha - 2s \dim X
	$$
	and it is obvious that this is positive if $\frac{r}{s}> \frac{2\dim X}{\alpha} $. In other words, if $\Delta'$ is sufficiently close to $\Delta$, then we get a contradiction. 
\end{proof}

\begin{prop}[Geography of weak lc models]\label{lem-geography-2}
Let $X$, $\{ \Delta_i \}$ and $\mathcal{P}$ be as in Lemma \ref{lem-geography-1}. Suppose that for each $\Delta \in \mathcal{P}$, either $(X,\Delta)$ has a good minimal model or a Mori fibre space. Then, the subset
$$
\mathcal{E}:=\{\Delta \in \mathcal{P}| \text{$K_X+\Delta$ is pseudo-effective.}  \}
$$
is a rational polytope contained in $\mathcal{P}$. Moreover, $\mathcal{E}$ admits a finite rational polyhedral decomposition $\mathcal{E}= \bigcup_k \mathcal{E}_k$ satisfying: 

$\bullet$ $\dim  \mathcal{E}_k =\dim  \mathcal{E}$ for every $k$;

$\bullet$ $\dim  \mathcal{E}_k \bigcap \mathcal{E}_l  <\dim  \mathcal{E}$ for every $l\neq k$;

$\bullet$ for any two divisors $\Delta$, $\Delta' \in \mathcal{E}$, $\Delta$ and $\Delta'$ belong to some same chamber if and only if there exists a normal variety $Y$ such that $(Y,\Delta_Y)$ and $(Y,\Delta_Y')$ are weak lc models of $(X,\Delta)$ and $(X,\Delta')$ respectively.
\end{prop}

\begin{proof}
	\emph{Step 1.}
By the proof of Lemma \ref{lem-geography-1}, possibly after shrinking $\mathcal{P}$ we can assume that $X$ is $\Q$-factorial klt. Let $\Delta$ be a point in $\mathcal{P}$. If $K_X+\Delta$ is pseudo-effective, then we show that there exists a rational polytope $\mathcal{E}_k$ containing $\Delta$ given by a common weak lc model. To this end, we run an LMMP on $K_X+\Delta$ with scaling. By Section \ref{preliminaries}: decreasing the coefficients appeared in LMMP with scaling and [\ref{B-lc-flips}, Theorem 1.9] the LMMP ends with a weak lc model $(Y,\Delta_Y)$ of $(X,\Delta)$ on which $K_Y+\Delta_Y$ is semi-ample. Note that $(Y,\Delta_Y)$ is not necessarily a log minimal model because $(X,\Delta)$ is not necessarily dlt, but we have $a(D,X,\Delta)< a(D,Y, \Delta_Y)$ for any exceptional$/Y$ prime divisor $D$ on $X$. We denote the birational contraction by $\phi :X \dashrightarrow Y$. It is easy to see that 
\begin{align*}
\mathcal{E}'&:=\{ \Delta' \in U| \text{$K_Y+\phi_\ast \Delta'$ is nef and} \\ 
& \text{for any prime divisor $D$ exceptional$/Y$, $a(D,X,\Delta')\le a(D,Y,\phi_\ast \Delta')$}  \}
\end{align*}
is a rational polytope by [\ref{B-II}, Proposition 3.2 (3)]. 

\emph{Step 2.}
We claim that, if $\dim \mathcal{E}' < \dim \mathcal{E}$, then we can extend $\mathcal{E}'$ so that $\dim \mathcal{E}' = \dim \mathcal{E}$. To this end, we first modify $\Delta \in \mathcal{E}'$ so that there is a neighborhood of $\Delta$ in $\mathcal{E}'$ sharing the same lc model. In fact, if we denote by $T$ the lc model of $(X,\Delta)$, then for any point $\Delta' \in \mathcal{E}'$ sufficiently close to $\Delta$, the lc model of $(X,\Delta')$ is either $T$ or a variety $T' \to T$. By an inductive argument we get a required boundary $\Delta$.

Pick a point $\Delta'' \in U$ sufficiently close to $\Delta$ such that $K_X+\Delta''$ is pseudo-effective and $\Delta'' \notin \mathcal{E}'$. For any exceptional$/Y$ prime divisor $D$, $a(D,X,\Delta'')< a(D,Y, \Delta_Y'')$ where $\Delta_Y''$ is the birational transform of $\Delta''$. By the construction of $\mathcal{E}'$, $K_Y+\Delta_Y''$ is not nef. But we can run an LMMP$/T$ on $K_Y+\Delta_Y''$ which terminates with a weak lc model $(W,\Delta_{W}'')$ as $\Delta_Y''$ is sufficiently close to $\Delta_Y$. To see this, because $\phi: X \dashrightarrow Y$ is a sequence of $K_X+\Delta''$-MMP, by assumption we can run an LMMP on $K_Y+\Delta_Y''$ which ends with a weak lc model $(W,\Delta_W'') \rightarrow T''$ of $(Y,\Delta_Y'')$ such that $K_W+\Delta_{W}''$ is semi-ample and $T''$ is the lc model. It is enough to prove that $g$ induces a morphism $T'' \rightarrow T$. Let $\widetilde{W}$ be a common log resolution of $(Y,\Delta_Y)$ and $(W,\Delta_W)$ and consider the following commutative diagram.
$$
\xymatrix{
	& \widetilde{W}\ar[dl]_{p}\ar[dr]^{q}  &\\
	Y \ar[d]_{g}\ar@{-->}[rr]_{\psi} &&  W\ar[d]^{g'} &\\
	T&& T'' \ar@{-->}[ll]^{}   &  } 
$$
By Lemma \ref{lem-geograpgy} (cf.[\ref{B-II}, Proposition 3.2 (5)]), one deduces that $\psi$ is $K_Y+\Delta_Y$-trivial since $\| \Delta_Y''-\Delta_Y\| \ll 1$, and hence $p^\ast(K_Y+\Delta_Y) =q^\ast (K_W+\Delta_W)$. In particular, we have $(g \circ p)^\ast A =q^\ast N$ for an ample divisor $A$ on $T$ and a nef divisor $N$ on $W$ which gives the morphism $W \to T$ by the Rigidity Lemma [\ref{Deb}, Lemma 1.15]. Since $K_W+\Delta_{W}'$ is semi-ample$/T$, we conclude that $T'' \dashrightarrow T$ is a morphism.

By the construction of $W$, we see that $(W,\Delta_W')$ is a weak lc model of every point $\Delta' \in \mathcal{E}'$ near $\Delta$ since $T$ is the common lc model of those boundaries $\Delta'$. So $(W,\Delta_W')$ is a weak lc model of every point $\Delta' \in \mathcal{E}'$. By replacing $Y$ with $W$, we extend $\mathcal{E}'$ so that it contains $\Delta'$. By continuing this process, we finally obtain a rational polytope $\mathcal{E}'$ with $\dim  \mathcal{E}' =\dim  \mathcal{E}$. Note that the way of extension is not unique, for example, when $\Delta$ lies on a face of $\mathcal{E}'$.

\emph{Step 3.}
Now we expand each $\mathcal{E}_k$ to its largest size. We show that for any pair of different chambers $\mathcal{E}_1$ and $\mathcal{E}_2$, their intersection has lower dimension. To see this, we assume the contrary. There is a point $\Delta \in \mathcal{E}_1$ but $\Delta \notin \mathcal{E}_2$, and $\dim \mathcal{E}_1 \bigcap \mathcal{E}_2 =\dim \mathcal{E}$. Pick an interior point $\Delta' \in \mathcal{E}_1 \bigcap \mathcal{E}_2 $. Suppose that $Y_1,Y_2$ give the chambers $\mathcal{E}_1,\mathcal{E}_2$ respectively, then for every point $\Delta''$ in $ \mathcal{E}_1 \bigcap \mathcal{E}_2$ on the segment defined by $\Delta,\Delta'$, log pairs $(Y_1,\Delta_{Y_1}''),(Y_2,\Delta_{Y_2}'')$ give the same log discrepancies which forces $\Delta \in  \mathcal{E}_1 \bigcap \mathcal{E}_2$. This is a contradiction.

 By a similar argument, we show that for any pair of points $\Delta$ and $\Delta'$, if there exists a normal variety $Y$ such that $(Y,\Delta_Y)$ and $(Y,\Delta_Y')$ are weak lc models of $(X,\Delta)$ and $(X,\Delta')$ respectively, then they belong to the same chamber. To see this, we assume $\Delta \in \mathcal{E}_1$ but $\Delta \notin \mathcal{E}_2$, while $\Delta' \in \mathcal{E}_2$ but $\Delta' \notin \mathcal{E}_1$. Then, for every point $\Delta''$ in $ \mathcal{E}_1$ on the segment defined by $\Delta,\Delta'$, log pairs $(Y_1,\Delta_{Y_1}''),(Y,\Delta_{Y}'')$ give the same log discrepancies which forces $\Delta' \in \mathcal{E}_1$. This is a contradiction.
 

\emph{Step 4.}
Finally we prove that the rational polyhedral decomposition $\mathcal{E}= \bigcup_k \mathcal{E}_k$ is finite. In particular, we obtain that $\mathcal{E}$ is a rational polytope. We proceed by induction on the dimension of $\mathcal{E}$. Assume the decomposition is not finite. Then, for each $k$ we pick an interior point $\Delta_k$. Since $\mathcal{P}$ is compact, we obtain a sequence of such points which converges to a point. By abuse of notation we again denote the sequence by $\{\Delta_k \}$ and the limit point by $\Delta$. It is obvious that $K_X+\Delta$ is pseudo-effective. Note that if $\Delta$ is not rational, then there is a rational polytope $\mathcal{E}_k$ containing $\Delta$ which is a contradiction. We therefore assume $\Delta$ is a $\Q$-divisor.

Now run an LMMP on $K_X+\Delta$ which terminates with a weak lc model $(Y,\Delta_Y)$. Then by Lemma \ref{lem-geograpgy} (cf.[\ref{B-II}, Proposition 3.2 (5)]), there is a small real number $\epsilon>0$ such that for every $\Delta' \in \mathcal{E}$ with $\| \Delta' -\Delta \| \le \epsilon$, any $K_Y+\Delta_Y'$-MMP is $K_Y+\Delta_Y$-trivial. So, for every $k\gg 0$, we have $(Y,\Delta_{k,Y})$ has a good minimal model. Moreover, if we run an LMMP on $K_Y+\Delta_{k,Y}$, then the sequence of LMMP is $K_Y+\Delta_Y$-trivial.

Let $\mathcal{R}_k$ be the ray from the point $\Delta$ passing through $\Delta_k$, and let $\mathcal{L}_k$ be the segment on $\mathcal{R}_k$ with $\| \Delta' -\Delta \| \le \epsilon$. We claim that $\mathcal{L}_k \subset \mathcal{E}_k$ for $k \gg 0$. In fact, for a point $\Delta_k' \in \mathcal{L}_k$ and its birational transform $\Delta_{k,Y}'$, we can run an LMMP on $K_Y+\Delta_{k,Y}'$ which terminates with a weak lc model $Y_k$. Since the LMMP above is $K_Y+\Delta_Y$-trivial, it is also an LMMP on $K_Y+\Delta_{k,Y}''$ for any point $\Delta_k'' \in \mathcal{L}_k \backslash \{ \Delta \}$ which implies the claim by Step 3.

Let $\epsilon' < \epsilon$ and let $\mathcal{S}$ be the surface defined by $\|\Delta'' -\Delta\| =\epsilon'$. We let $\Delta''_k$ be the intersection point of $\mathcal{S}$ and $\mathcal{L}_k$ for every $k \gg 0$. It is clear that $\dim \mathcal{S}\bigcap \mathcal{E} < \dim \mathcal{E}$ and that $\Delta_k''$ is an interior point of a chamber. By induction this is a contradiction. 
\end{proof}

By combining Theorem \ref{t-main-1} and Lemma \ref{lem-geography-2}, we immediately obtain the following corollary.

\begin{cor}\label{cor-geography}
Let $X$, $\{ \Delta_i \}$ and $\mathcal{P}$ be as in Lemma \ref{lem-geography-1}. Given a surjective morphism $f:X \rightarrow Z$. Assume further that $K_X+\Delta_i \sim_{\mathbb{R}} 0/Z$ for every index $i$ and that $\Delta_i \ge A$ where $A\sim_\R f^{\ast} A_Z$ is the pull-back of an ample divisor $A_Z$ on $Z$. Then, the subset
$$
\mathcal{E}:=\{\Delta \in \mathcal{P}|\text{ $K_X+\Delta$ is pseudo-effective.}  \}
$$
is a rational polytope contained in $\mathcal{P}$. Moreover, $\mathcal{E}$ admits a finite rational polyhedral decomposition $\mathcal{E}= \bigcup_k \mathcal{E}_k$ satisfying: 

$\bullet$ $\dim  \mathcal{E}_k =\dim  \mathcal{E}$ for every $k$;

$\bullet$ $\dim  \mathcal{E}_k \bigcap \mathcal{E}_l  <\dim  \mathcal{E}$ for every $l\neq k$;

$\bullet$ for any two divisors $\Delta$, $\Delta' \in \mathcal{E}$, $\Delta$ and $\Delta'$ belong to some same chamber if and only if there exists a normal variety $Y$ such that $(Y,\Delta_Y)$ and $(Y,\Delta_Y')$ are weak lc models of $(X,\Delta)$ and $(X,\Delta')$ respectively.
\end{cor}

To establish the terminations for $\Q$-factorial lc pairs with boundaries containing ample divisors, we begin with the following lemma. 

\begin{lem}\label{lem-polytope}
	Let $(X,B)$ be a $\Q$-factorial projective lc pair and $f: X \dashrightarrow X'$ be a step of LMMP on some log canonical divisor $K_X+\Delta$ which is $K_X+B$-non-positive. Suppose there is a rational polytope $\mathcal{P} \subset \mathrm{Div}_\R(X)$ containing $K_X+B$ as an interior point such that $\mathcal{E} \subset \mathcal{P}$ has a rational polyhedral decomposition on which the asymptotic vanishing orders given by discrete valuations are piecewise linear. Then there is a rational polytope $\mathcal{P}' \subset \mathrm{Div}_\R(X')$ containing $K_{X'}+B'$ as an interior point satisfying the same property, where $B'$ is the birational transform of $B$. Moreover, if $\dim \varphi (\mathcal{P}) =\dim N^1(X)$ where $\varphi: \mathrm{Div}_\R(X) \to N^1(X)$ is the natural projection, then  $\dim \varphi' (\mathcal{P}') =\dim N^1(X')$  where $\varphi': \mathrm{Div}_\R(X') \to N^1(X')$ is the natural projection.
\end{lem}	
\begin{proof}
	If $f$ is a log flip, then the result follows easily from the fact $f$ being an isomorphism in codimension one. So we suppose $f$ is a divisorial contraction. Let $V$ be a sufficiently large finite dimensional subsapce of $\mathrm{Div}_\R(X)$ containing $\mathcal{P}$ and let $H \subset V$ be a hyperplane defined by $\Gamma=0$ where $\Gamma$ denotes the extremal ray contracted in the step. By definition of asymptotic vanishing orders given by discrete valuations, we see there is the one-to-one correspondence between $H$ and $V'$ where $V' \subset \mathrm{Div}_\R(X')$ is the image of $V$ under $f_*$. 
	
	Now we pick a rational sub-polytope $\mathcal{Q} \subset \mathcal{P}$ of codimension one which contains $K_X+B$ as an interior point such that the natural projection from the half-space $\Gamma \le 0$ to the hyperplane $H$ gives the one-to-one correspondence between $\mathcal{Q}$ and its image $\mathcal{Q}_H$ on $H$. Let $\mathcal{P}' \subset V'$ be the birational transform of $\mathcal{Q}$ on $X'$. It is easy to check that $\mathcal{P}'$ satisfies the required properties.   
\end{proof}

\begin{thm}\label{t-main-2}
Let $(X,\Delta)$ be a $\Q$-factorial projective lc pair such that $\Delta \geq A$ where $A \geq 0$ is an ample divisor. Then, any LMMP on $K_X+\Delta$ with scaling terminates. 
\end{thm}
\begin{proof}
	Run an LMMP on $K_X+\Delta$ with scaling of $C$. Suppose that there exists an infinite sequence of LMMP
	$$
	(X,\Delta) = (X_1,\Delta_1)\dashrightarrow \cdots \dashrightarrow (X_i,\Delta_i) \dashrightarrow (X_{i+1},\Delta_{i+1}) \dashrightarrow \cdots
	$$
	where $X_i \dashrightarrow X_{i+1}$ is a divisorial contraction or a log flip for every $i$. 
	We write $\lambda_{i} \le 1$ as the coefficients appeared in the LMMP and $\lambda= \lim\limits_{i \rightarrow \infty}\lambda_i$. The strategy to show the contradiction is borrowed from [\ref{Corti-Lazic}, Theorem 6.5].
	
	
	Since $\Delta \geq A$, by modification of coefficients we can assume $(X,(1+\epsilon)(\Delta+C))$ is lc for a sufficiently small number $\epsilon>0$. We denote by $\varphi: \mathrm{Div}_\R(X) \rightarrow N^1(X)_\R$ and $\varphi_i: \mathrm{Div}_\R(X_i) \rightarrow N^1(X_i)_\R$ the natural
	projections. Pick $\Q$-boundaries $\Delta^{(1)}$, $\ldots$, $\Delta^{(r)}$ on $X$ such that:
	
	(i) $\Delta + \lambda_1 C$ is an interior point of the rational polytope $\mathcal{P}$ defined by $\{ \Delta^{(k)} \}$ and $\Delta$,
	
	(ii) the dimension of the polytope $\varphi(\mathcal{P}) \subset N^1(X)_\R$
	is $\dim N^1(X)_\R$,
	
	(iii) $(X,\Delta^{(k)})$ is lc, and $\| \varphi(K_X+\Delta + \lambda_1  C) - \varphi(K_X+\Delta^{(k)} )\| \ll 1$ for every $k$.

	Next we define $\mathcal{P}_i \subset N^1(X_i)_\R$ inductively. Suppose we are given $\mathcal{P}_i$ and a step of LMMP $\phi_i:X_i \dashrightarrow X_{i+1}$. If $\phi_i$ is a log flip, then we let $\mathcal{P}_{i+1}$ be the birational transform of $\mathcal{P}_i$. If $\phi_i$ is a divisorial contraction, then we let $\mathcal{Q}_i$ be the intersection of $\mathcal{P}_i$ with the half-space $\Gamma \le 0$ where $\Gamma$ is the extremal ray contracted by $\phi_i$, and let $\mathcal{P}_{i+1}$ be the birational transform of $\mathcal{Q}_i$. With this inductive construction, it is obvious that $\dim \varphi_i(\mathcal{P}_i) =\dim N^1(X_i)_\R$. Note that $\Delta_i + \lambda_i  C_i$ is not necessarily an interior point of $\mathcal{P}_i$. However, By Lemma \ref{lem-polytope} we still have that $\varphi(K_{X_i}+\Delta_i + \lambda_i  C_i)$ is an interior point of $\varphi_i(K_{X_i}+\mathcal{P}_i)$.
	
	By Corollary \ref{cor-geography}, there is a rational polytope $\mathcal{E} \subset \mathcal{P} $ which admits a finite polyhedral decomposition $\mathcal{E} = \bigcup_k \mathcal{E}^{(k)} $ satisfying the conditions listed in Corollary \ref{cor-geography}. Because $\Delta + \lambda_1  C $ is an interior point, by Condition (ii) there exists a chamber $\mathcal{E}^{(k_1)} $ containing this point such that $\varphi (K_{X }+\mathcal{E}^{(k_1)} )$ is contained in the nef cone of $N^1(X)_\R$. In particular, $\varphi (K_{X }+\mathcal{E}^{(k_1)} )$ contains an ample divisor.
	
	Let $\mathcal{E}_i$ and $\mathcal{E}^{(k)}_i$ be constructed in the same way as in the third paragraph. Also, let the decompositions of $\mathcal{E}_i= \bigcup_k \mathcal{E}^{(k)}_i$ be constructed in the same way. We may lose the log canonicity of $(X_i,\Delta'^{(k)}_i)$ but the linearity of asymptotic vanishing orders given by discrete valuations on each chamber will be preserved by Lemma \ref{lem-polytope} (cf.[\ref{Corti-Lazic}, Lemma 5.2]). Since $\lambda_i >0$, we see $\varphi_i(\Delta_i + \lambda_i C_i)$ is an interior point of $\varphi_i(\mathcal{P}_i)$ and there exists a chamber $\mathcal{E}^{(k_i)}_i$ containing $\Delta_i + \lambda_i C_i$ such that $\varphi_i(K_{X_i}+\mathcal{E}^{(k_i)}_i)$ is contained in the nef cone. In particular by Condition (ii), $\varphi_i(K_{X_i}+\mathcal{E}^{(k_i)}_i)$ contains an ample divisor. 
	
	Let $i_0$ be an index such that the sequence of LMMP consists of only log flips after $X_{i_0}$. Since the decomposition is finite, there exists two indices $i,j\ge i_0$ such that $k_i=k_j$ but $i \neq j$. Then by the Rigidity Lemma [\ref{Deb}, Lemma 1.15] $X_i$ and $X_j$ are isomorphic. This is a contradiction.
\end{proof}

\begin{rem}
	It is worthy to note that a log flip on a non-$\Q$-factorial lc pair sometimes increases the Picard number. For example, see [\ref{Fujino-3}, Example 7.5.1]. Therefore the argument above cannot be applied directly to an LMMP for non-$\Q$-factorial pairs. To generalise the previous theorem we need a detailed analysis on the structure of $N^1(X_i)_\R$.
\end{rem}

\begin{cor}\label{w-Mori}
Let $(X,B)$ be a $\Q$-factorial projective lc pair such that $K_X+B$ is not pseudo-effective. Then, any LMMP on $K_X+B$ with scaling of some ample divisor terminates with a weak Mori fibre space.
\end{cor}

\begin{proof}
Pick an ample divisor $H$ and run an LMMP on $K_X+\Delta$ with scaling of $H$. Note that this LMMP is also an LMMP on $K_X+\Delta+\epsilon H$ with scaling of $H$ where $\epsilon>0$ is a small number. Now the conclusion follows from Theorem \ref{t-main-2}.	
\end{proof}

\vspace{0.3cm}
\section{Appendix}\label{appendix}
\begin{center}
Written by Jinhyung Park
\end{center}

The aim is to show the finite generation of Cox rings of log canonical Fano pairs, which was conjectured by Cascini and Gongyo ([\ref{CG}, Conjecture 4.1]).

\begin{thm}\label{coxlcfano}
Let $X$ be a $\Q$-factorial projective variety. Assume there is a boundary $\Delta$ such that $(X,\Delta)$ is log canonical and that $-(K_X+B)$ is ample. Then a Cox ring of $X$ is finitely generated, and hence $X$ is a Mori dream space.
\end{thm}

The following thereom is a key ingredient of the proof of Theorem \ref{coxlcfano}.

\begin{thm}\label{lcmmp}
	Let $(X, \Delta+A)$ be a projective log canonical pair such that $\Delta$ is a boundary divisor and $A \geq 0$ is an ample divisor. If $K_X+\Delta+A$ is a pseudo-effective $\Q$-divisor, then $R(X, K_X+\Delta+A)$ is finitely generated.
\end{thm}

\begin{proof}
	Since $(X, \Delta+A)$ has a good minimal model by Corollary \ref{t-main-cor'}, it follows that $R(X, K_X+\Delta+A)$ is finitely generated.
\end{proof}

We also need an easy lemma (cf. [\ref{CG}, Propositions 2.9 and 2.10]).

\begin{lem}\label{bundlelc}
	Let $(X, \Delta)$ be a projective log canonical pair such that $-(K_X+\Delta)$ is ample, and $L_1, \ldots, L_m$ be Cartier divisors on $X$. Let $E:= \bigoplus_{i=1}^m \mathcal{O}_X(L_i)$, and $\pi \colon Y=\mathbb{P}(E) \to X$ be the corresponding projective bundle. Then there is a boundary divisor $\Gamma$ on $Y$ such that $(Y, \Gamma)$ is a projective log canonical  pair and $-(K_Y+\Gamma)$ is ample. 
\end{lem}

\begin{proof}
	After possibly tensoring by a suitable ample divisor on $X$, we may assume that $L_1, \ldots, L_m$ are very ample divisors on $X$. Then the tautological section $H$ of $E$ is an ample divisor on $Y$. 
	Now fix a sufficiently small rational numver $\epsilon >0$ such that $-(K_X+\Delta) - \epsilon \sum_{i=1}^m L_i$ is ample. We can take an ample divisor $A$ such that  $A \sim_{\Q} -(K_X+\Delta)- \epsilon \sum_{i=1}^m L_i$ and $(X, \Delta+A)$ is a projective log canonical pair. 
	Let $T_1, \ldots, T_m$ be the divisors on $Y$ given by the summands of $E$, and $T:=\sum_{i=1}^m T_i$. Note that $T \sim mH - \sum_{i=1}^m \pi^*L_i$.
	There exists a boundary divisor $\Gamma$ on $Y$ such that 
	$(Y, \Gamma)$ is a projective log canonical pair and
	\begin{align*}
\Gamma &\sim_{\Q} (1-\epsilon)T + \pi^*(\Delta+A) \\
&\sim_{\Q} T - \epsilon\left(mH - \sum_{i=1}^m \pi^* L_i\right) + \pi^*(\Delta+A) \\
	& \sim_{\Q} T + \pi^*(-K_X) - \epsilon mH.
	\end{align*}
	Since $-K_Y = T - \pi^*K_X$, it follows that
	$
	-(K_Y+\Gamma) \sim_{\Q}  -K_Y-T + \pi^*K_X + \epsilon m H\sim_{\Q} \epsilon mH
	$
	is ample.
\end{proof}

We then prove Theorem \ref{coxlcfano}.

\begin{proof}[Proof of Theorem \ref{coxlcfano}]
	By Kodaira vanishing theorem for log canonical pairs, we have $H^1(X, \mathcal{O}_X)=0$, and thus $\Pic_{\Q}(X)$ is a finite dimensional vector space over $\Q$. Choose a basis $L_1, \ldots, L_m$ of $\Pic_{\Q}(X)$ such that each $L_i$ is a Cartier divisor on $X$ and the convex hull of $L_1, \ldots, L_m$ in $N_{\R}(X)$ contains the effective cone $\mathrm{Eff}(X)$. Let $Y:=\mathbb{P}( \bigoplus_{i=1}^m \mathcal{O}_X(L_i))$, and $H$ be the tautological section of $\bigoplus_{i=1}^m \mathcal{O}_X(L_i)$.
	Note that $H$ is a big Cartier divisor on $Y$.
	It is easy to check that a Cox ring of $X$ is finitely generated if and only if the section ring $R(Y, H)$ is finitely generated.
	By Lemma \ref{bundlelc}, there is a boundary divisor $\Gamma$ on $Y$ such that $(Y, \Gamma)$ is a projective log canonical pair and $-(K_Y+\Gamma)$ is ample. We can choose an ample divisor $A$ on $Y$ such that $(Y, \Gamma+A)$ is a projective log canonical pair and $K_Y+\Gamma+A \sim_{\Q} \epsilon H$ for a sufficiently small rational number $\epsilon >0$.
	By Theorem \ref{lcmmp}, $R(Y, \epsilon H)$ is finitely generated, and so is $R(Y, H)$. Therefore, a Cox ring of $X$ is also finitely generated.
\end{proof}

Now we discuss about the characterization problem for $\Q$-factorial log canonical weak Fano pairs whose Cox rings are finitely generated.
Note that
$$
\begin{array}{l}
\{X \mid (X, \Delta) \text{ is a $\Q$-factorial klt Fano pair} \}  \\
= \{X \mid (X, \Delta) \text{ is a $\Q$-factorial klt weak Fano pair} \}
\end{array}
$$
and every element of the above set is a Mori dream space by [\ref{BCHM}, Corollary 1.3.2]. However, in the log canonical weak Fano pair case, we only have strict inclusions
$$
\begin{array}{l}
\{X \mid (X, \Delta) \text{ is a $\Q$-factorial log canonical Fano pair} \} \\
\subsetneq \{X \mid (X, \Delta) \text{ is a $\Q$-factorial log canonical weak Fano pair and $X$ is a Mori dream space} \} \\
\subsetneq \{X \mid (X, \Delta) \text{ is a $\Q$-factorial log canonical weak Fano pair} \}
\end{array}
$$
by the following example.

\begin{exa}
	Let $Y$ be a log canonical non-rational del Pezzo surface, and $f : X \to Y$ be the minimal resolution. We denote by $-K_X=P+N$ the Zariski decomposition. 
	Then $(X, N)$ is a log canonical weak Fano pair, but $X$ is not a Mori dream space since $\mathrm{Pic}(X)$ is not finitely generated.
	Note that $N$ is an elliptic curve contracted by $f$ (see [\ref{HP}, Theorem 1.6]). Take a blow-up $g : X' \to X$ at a point lying on $N$, and then contract $g_*^{-1} N$ so that we get a normal projective surface $Y'$. 
	Then $Y'$ is a log canonical weak del Pezzo surface, $Y'$ is a Mori dream space ([\ref{HP}, Corollary 1.9]), and there is no boundary divisor $\Delta$ with $(Y', \Delta)$ being a log canonical del Pezzo pair.
\end{exa}

Recall that if $(X, \Delta)$ is a $\Q$-factorial log canonical Fano pair, then $X$ is rationally chain connected ([\ref{HM}, Corollary 1.2]). Furthermore, when $(X, \Delta)$ is a log canonical weak del Pezzo pair, $X$ is a Mori dream space if and only if $X$ is rationally chain connected ({[\ref{HP}, Corollary 1.9]).
	Thus it is tempting to expect that for a $\Q$-factorial log canonical weak Fano pair $(X, \Delta)$, the finite generation of Cox ring of $X$ is equivalent to the rationally chain connectedness of $X$. However, this naive expectation turns out to be false in both directions by the following examples.
	
\begin{exa}
		Let $S$ be a Mori dream K3 surface, and $A$ be a very ample divisor on $S$. Let $X:=\mathbb{P}(\mathcal{O}_S \oplus \mathcal{O}_S(-A))$, and $H$ be the tautological section of $\mathcal{O}_S \oplus \mathcal{O}_S(-A)$. 
		Then $(X, H)$ is a log canonical weak Fano pair and $X$ is a Mori dream space by [\ref{CG}, Proposition 2.6], but $X$ is not rationally chain connected.
\end{exa}
	
\begin{exa}
		Let $S$ be the blow-up of $\mathbb{P}^2$ at 9 very general points, and $A$ be a very ample divisor on $S$ such that $A-K_S$ is also very ample. Let $X:=\mathbb{P}(\mathcal{O}_S(-A) \oplus \mathcal{O}_S (-K_S))$, and $H$ be the tautological section of $\mathcal{O}_S(-A) \oplus \mathcal{O}_S (-K_S)$.
		Then  $(X, H)$ is a log canonical weak Fano pair and $X$ is rationally connected, but $X$ is not a Mori dream space. 
\end{exa}

\vspace{0.3cm}

\vspace{2cm}

\flushleft{National} Center of Theoretical Sciences,\\
National Taiwan University,\\
Roosevelt Road,\\
Taipei, 10617,\\
Taiwan\\
email: zyhu@ncts.ntu.edu.tw

\vspace{1cm}

\end{document}